%% file: main.tex
\documentclass[11pt]{article}
\pdfoutput=1
\usepackage[margin=1in]{geometry}
\input{preamble}

\usepackage{algorithm}
\usepackage{algorithmic}

\title{Optimization on Pareto sets: \\ On a theory of multi-objective optimization}

\author{%
  Abhishek Roy* \\ \texttt{a2roy@ucsd.edu} 
  \and
  Geelon So* \\
  \texttt{geelon@ucsd.edu} \\
  University of California, San Diego\\
  La Jolla, CA 92093
  \and 
  Yi-An Ma\\
  \texttt{yianma@ucsd.edu}
}

\begin{document}
\maketitle

\begin{abstract}
In multi-objective optimization, a single decision vector must balance the trade-offs between many objectives. Solutions achieving an optimal trade-off are said to be Pareto optimal---these are decision vectors for which improving any one objective must come at a cost to another. But as the set of Pareto optimal vectors can be very large, we further consider a more practically significant \emph{Pareto-constrained optimization problem}, where the goal is to optimize a preference function constrained to the Pareto set. 

We investigate local methods for solving this constrained optimization problem, which poses significant challenges because the constraint set is (i) implicitly defined, and (ii) generally non-convex and non-smooth, even when the objectives are. We define notions of optimality and stationarity, and provide an algorithm with a last-iterate convergence rate of $O( K^{-1/2})$ to stationarity when the objectives are strongly convex and Lipschitz smooth. 
\end{abstract}

\input{body}

\newpage
\input{appendix}

\section*{Acknowledgements}
This work is supported in part by the National Science Foundation Grants NSF-SCALE MoDL(2134209) and NSF-CCF-2112665 (TILOS), the U.S. Department Of Energy, Office of Science, and the Facebook Research award.

\bibliography{references}

\end{document}

%% file: preamble.tex
\usepackage{xcolor}
\usepackage[hypertexnames=false]{hyperref}
\hypersetup{
    colorlinks,
    linkcolor={red!50!black},
    citecolor={blue!50!black},
    urlcolor={blue!80!black}
}
\usepackage{ifthen}
\usepackage{amsthm}
\usepackage{amsmath,amssymb,amsfonts,bbm,dsfont}
\usepackage{mathtools}
\usepackage{array}
\usepackage{caption}
\usepackage{censor}
\usepackage{booktabs}
\usepackage{nicefrac}      
\usepackage{microtype}
\usepackage{afterpage}
\usepackage{inconsolata}

\usepackage[utf8]{inputenc} 
\usepackage[T1]{fontenc}    
\usepackage{booktabs}       
\usepackage{nicefrac}       
\usepackage{microtype}      

\usepackage{xcolor}
\usepackage{placeins}
\usepackage{pgfplots}
\pgfplotsset{compat=1.18}
\usepackage{tikz}
\usetikzlibrary{calc}
\usetikzlibrary{patterns,patterns.meta}
\pgfdeclarepatternformonly{south east lines}{\pgfqpoint{-0pt}{-0pt}}{\pgfqpoint{3pt}{3pt}}{\pgfqpoint{3pt}{3pt}}{
    \pgfsetlinewidth{0.4pt}
    \pgfpathmoveto{\pgfqpoint{0pt}{3pt}}
    \pgfpathlineto{\pgfqpoint{3pt}{0pt}}
    \pgfpathmoveto{\pgfqpoint{.2pt}{-.2pt}}
    \pgfpathlineto{\pgfqpoint{-.2pt}{.2pt}}
    \pgfpathmoveto{\pgfqpoint{3.2pt}{2.8pt}}
    \pgfpathlineto{\pgfqpoint{2.8pt}{3.2pt}}
    \pgfusepath{stroke}}
\usetikzlibrary{positioning,cd}

\definecolor{UCnavyy}{rgb}{0.094117, 0.168627, 0.2862745}           
\definecolor{navy}{rgb}{0, 0.415686, 0.588235}  

\definecolor{cf9f9f9}{RGB}{249,249,249}
\definecolor{cb3b3b3}{RGB}{179,179,179}
\definecolor{c808080}{RGB}{128,128,128}
\definecolor{c1a1a1a}{RGB}{26,26,26}
\definecolor{cffffff}{RGB}{255,255,255}

\usepackage{flafter}
\usepackage{comment}
\usepackage{cleveref}

\usepackage{natbib}
\bibpunct{(}{)}{;}{a}{,}{,}

\usepackage{thmtools,thm-restate}
\newtheorem{theorem}{Theorem}

\newtheorem{lemma}{Lemma}

\newtheorem{example}{Example}
\newtheorem{proposition}{Proposition}

\newtheorem*{assumption*}{Assumption}
\newtheorem{assumption}{Assumption}

\newtheorem{definition}{Definition}
\newtheorem{remark}{Remark}

\usepackage{apptools}
\AtAppendix{\counterwithin{theorem}{section}}
\AtAppendix{\counterwithin{conjecture}{section}}
\AtAppendix{\counterwithin{exercise}{section}}
\AtAppendix{\counterwithin{lemma}{section}}
\AtAppendix{\counterwithin{fact}{section}}
\AtAppendix{\counterwithin{claim}{section}}
\AtAppendix{\counterwithin{example}{section}}
\AtAppendix{\counterwithin{proposition}{section}}
\AtAppendix{\counterwithin{corollary}{section}}
\AtAppendix{\counterwithin{problem}{section}}
\AtAppendix{\counterwithin{assumption}{section}}

\AtAppendix{\counterwithin{definition}{section}}
\AtAppendix{\counterwithin{remark}{section}}
\AtAppendix{\counterwithin{observation}{section}}

\AtAppendix{\counterwithin{exercise}{section}}

\renewcommand{\epsilon}{\varepsilon}

\DeclareMathOperator*{\argmin}{\mathrm{arg\,min}}

\DeclareMathOperator*{\minimize}{\mathrm{minimize}}

\newcommand{\norm}[1]{\lVert#1\rVert}

\newcommand\numberthis{\addtocounter{equation}{1}\tag{\theequation}}  

\usepackage{mdframed}
\newmdenv[
  topline=false,
  bottomline=false,
  rightline=false,
  skipabove=\topsep,
  skipbelow=\topsep,
  innertopmargin=0pt,
  innerbottommargin=0pt
]{siderules}

\usepackage{ragged2e}

%% file: body.tex
\section{Introduction}
The theory of optimization has provided the foundations for analyzing large-scale machine learning, giving us a language for understanding not only training accuracy, but also generalization~\citep{hardt2022patterns} and adaptive decision making~\citep{Puterman94MDP}. However, in practice, we often need to further account for additional desiderata: resource constraints, fairness, fine-tunability, and so on. As a result, \emph{multi-objective optimization} (MOO) has increasingly drawn interest from the machine learning community, since it naturally generalizes the single objective paradigm of classical learning while also being able to attend to these additional requirements. 

Examples of machine learning settings formulated as MOO problems include those with multiple tasks~\citep{sener2018multi,doersch17multi}, different data distributions~\citep{dong2015multi,huang2015rapid}, fairness requirements~\citep{martinez20minimax,la2023optimizing,kamani2021pareto}, inverse reinforcement learning~\citep{pirotta16inverse}, and the need to balance compute and power consumption among multiple algorithmic modules~\citep{ghosh13towards}. 

The solutions to MOO problems are those that achieve optimal trade-offs, or \emph{Pareto optimality}; together, they form the \emph{Pareto set}. But because the Pareto set generally does not contain a single solution, there is a need to make a further selection from the Pareto optimal solutions. Currently, there are two main approaches to making this selection. The first is to find a representative subsample of the Pareto set: this reduces the number of solutions that need to be inspected before making a final decision~\citep{lin19pareto,liu21profiling,kobayashi2019bezier,guerreiro21hypervolume}. The other approach is to scalarize the multiple objectives into a single objective, say, by taking a linear combinations of the objectives~\citep{mahapatra2020multi}.

However, as the number of objectives and dimensions increase, the Pareto set can become extremely large, forcing the size of a representative subsample to also become untenably large. Furthermore, the geometry of the Pareto set can be quite complicated, with ``needle-like extensions'' and ``knees'' \citep{kulkarni2022regularities}, which poses difficulties for sampling. Even with quadratic objectives in two dimensions, we can observe singularities in the Pareto set, see \cite{sheftel2013geometry} or \Cref{fig:hard-pareto}. The other scalarization approach is also not without difficulties. As the objective functions can be incomparable, it can be challenging to find a meaningful weighting of the objectives.

For a more principled selection, we assume that we have an additional \emph{preference function} $f_0$, which we aim to optimize constrained to the Pareto set. In supervised learning tasks, this preference function is oftentimes the loss function of a generic dataset. In economic and decision making problems, it is usually taken to be the social welfare of the entire community of users. This approach has been considered in various contexts such as portfolio management \citep{thach1996dual} and manufacturing planning \citep{yamamoto2002optimization}, in addition to machine learning and optimization~\citep{ye2022pareto}. While heuristics have been proposed, little is known about the convergence properties of these algorithms. This prompts us to ask:

\begin{center}
{\it Given a set of objective functions $(f_1,\dots,f_n)$ and a preference function $f_0$, what is a suitable approximate solution concept and what are efficient algorithms to achieve it?}
\end{center}

\subsection{Main results}
In MOO, we are given a set of $n$ objective functions $F \equiv (f_1,\ldots, f_n) : \mathbb{R}^d \to \mathbb{R}^n$ that are jointly optimized over a shared decision space $\mathbb{R}^d$:
\begin{equation} \label{eqn:moo}
\minimize_{x \in \mathbb{R}^d}\quad F(x).
\end{equation}
The solution concept for (\ref{eqn:moo}) is typically defined as the set of Pareto optimal solutions, $\mathrm{Pareto}(F)$, which consists of decision vectors $x \in \mathbb{R}^d$ that make an optimal trade-off between objectives. And to further decide which trade-off to make, we consider the \emph{Pareto-constrained optimization problem}, in which the aim is to optimize a preference function $f_0 : \mathbb{R}^d \to \mathbb{R}$ constrained to the Pareto set of $F$:
\begin{equation}  \label{eqn:pso}
\minimize_{x \in \mathrm{Pareto}(F)}\quad f_0(x).
\end{equation}

\begin{figure}[t]
    \centering
    \captionsetup{width=0.9\textwidth}
    \input{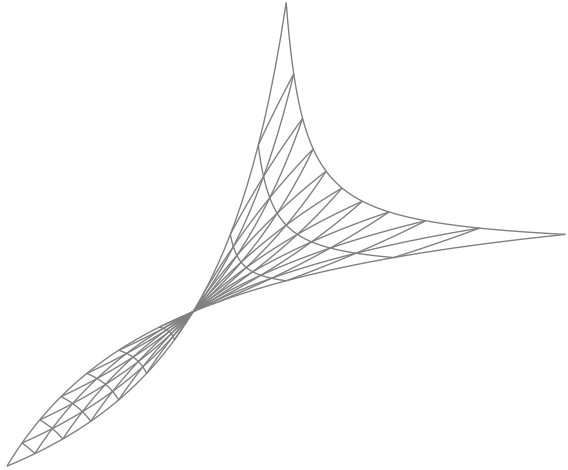}
    \caption{A Pareto set for three positive-definite quadratic objective functions in $\mathbb{R}^2$. 
    The grid lines show the coordinate maps for $x^* : \Delta^{2} \to \mathrm{Pareto}(f_1, f_2, f_3)$, where $\Delta^{2}$ is the $3$-simplex.
    Even in this well-structured setting, the Pareto set is not convex or smooth.}
    \label{fig:hard-pareto}
\end{figure}

This problem is challenging not only because the constraint set is defined implicitly as the solution to the MOO problem from \Cref{eqn:moo}, but because it is also non-convex and non-smooth. Even in the case of linear preference functions, the problem is known to be NP-hard \citep{fulop1993equivalence}. In fact, it is not obvious how to even define an appropriate relaxation of the problem such as stationarity that can be attained through optimization, given the challenges of non-convex non-smooth optimization \citep{zhang2020complexity,kornowski2021oracle,li2020understanding,jordan2023deterministic}. However, we show in this work that the Pareto set has additional geometry when the objectives are strongly convex that allows us to relax the Pareto-constrained optimization problem to a strong notion of stationarity that is necessary for optimality and that can be efficiently attained:
\begin{enumerate}
    \item We show that the Pareto-constrained optimization problem has an equivalent reformulation as a smooth optimization problem over a linear constraint set (\Cref{prop:lifted-pareto-set}). This allows us to introduce solution concepts such as (approximate) \emph{preference stationarity} in the standard way. Furthermore, we show that the solution concepts are geometrically meaningful (\Cref{prop:approx-stationary}).
    \item While the reformulation solves the issue of non-convexity and non-smoothness, the reformulated objective function remains implicit, which can make it hard to design optimization methods and provide simple analyses. If the objectives and preference are sufficiently smooth (Assumptions~\ref{ass:objectives}--\ref{ass:f0}), we construct a family of upper bounds for the reformulated objective function (\Cref{prop:maj-sur}), providing a general framework to analyze iterative gradient-based methods.
    \item We provide the \emph{Pareto majorization-minimization} algorithm, which iteratively (i) computes these upper bounds and (ii) minimizes them. In our setting, this amounts to solving a sequence of (i) unconstrained strongly-convex optimization problems and (ii) quadratic programs. We show that it suffices to solve the strongly convex programs up to $O(\epsilon_0^2)$-optimality and the quadratic programs up to $O(\epsilon_0)$-stationarity. Then, no more than $O(\epsilon_0^{-2})$ rounds of optimization are needed to attain an $\epsilon_0$-approximate preference stationary solution (\Cref{thm:pmm-convergence}).
\end{enumerate}

\subsection{Related work}
Selecting a single decision out of all Pareto optimal decisions is a fundamental problem of MOO that does not appear in the classical single-objective setting; in MOO, there is no canonical total ordering of the solutions \citep{miettinen1999nonlinear}. Broadly, the approaches to making such a selection can be categorized as \emph{a~priori} and \emph{a~posteriori} \citep{hwang2012multiple}.\footnote{They also include two other categories: the \emph{no-preference} and \emph{interactive} approach. In the former, any Pareto optimal decision will do, while in the latter, candidates are presented adaptively to an interactive decision maker.} 

In the \emph{a~priori} setting, the preferences of the decision maker is known beforehand. While in the \emph{a~posteriori} approach, the goal is to present a decision maker with a representative spread of Pareto optimal options, from which the decision maker will make a final decision. But because the Pareto set can become very high-dimensional, the \emph{a posteriori} approach becomes less viable (or needs to become more interactive) as the number of objectives and dimensions increase.

Instead, we work in the \emph{a priori} setting and consider optimization constrained to the Pareto set, also sometimes called \emph{semivectorial bilevel optimization} or \emph{optimization on efficient sets}, which can be considered an instantiation of bi-level optimization \citep{yamamoto2002optimization,bonnel2006semivectorial,dempe2018bilevel}. This problem is known to be NP-hard \citep{fulop1993equivalence} and algorithms for solving this problem tend to focus on settings with: (i) linear preference functions \citep{philip1972algorithms,benson1984optimization,liu2018primal}, (ii) linear objectives \citep{dauer1991optimization,bolintineanu1993minimization,tao1996numerical,yamamoto2002optimization}, or (iii) specific choices of preference functions such as the Tchebycheff weighting function \citep{steuer1989tchebycheff}. To our knowledge, the only other algorithmic work that considers the general problem with nonlinear objectives is \cite{ye2022pareto}. 

However, the stationary condition introduced by \cite{ye2022pareto}, defined as stationarity with respect to the proposed optimization dynamics, does not have a clear connection to preference optimality. In fact, as it is a non-trivial first-order stationary condition, the stationarity notion defined therein is not a necessary condition (\Cref{prop:impossibility}); there are settings where such dynamics avoid optimal points (see \Cref{ex:png}).

We are able to introduce a simple and necessary condition for preference optimality by making use of the manifold structure of the Pareto set. While its smooth structure has previously been recognized \citep{hillermeier2001generalized,hamada2020topology}, the prior focus has been on the extrinsic Pareto manifold within an ambient space, from which it inherits its smoothness. We take a different approach and work with the Pareto manifold intrinsically. Since it is diffeomorphic to the simplex, conceptually, this greatly simplifies optimization constrained to the Pareto set. And in order to overcome the implicit nature of the Pareto manifold, we use ideas from majorization-minimization  and trust-region approaches to optimization, where approximate gradient information can be used to make provable improvements \citep{lange2000optimization,marumo2023majorization}.

\section{Preliminaries} \label{sec:preliminaries}
Let $f_1, \ldots, f_n : \mathbb{R}^d \to \mathbb{R}$ be objective functions, $f_0 : \mathbb{R}^d \to \mathbb{R}$ be a preference function. We assume:
\begin{enumerate}
    \item[(\ref{ass:objectives})] The objectives are strongly convex and twice-differentiable with Lipschitz-continuous gradients. 
    \item[(\ref{ass:hesssmooth})] The objectives have Lipschitz-continuous Hessians.
    \item[(\ref{ass:f0})] The preference has Lipschitz-continuous gradients.
\end{enumerate}
Let $[n] := \{1,\ldots, n\}$. We denote the $(n-1)$-simplex by $\Delta^{n-1}$, which is the set of convex weights:
\[\Delta^{n-1} := \bigg\{\beta \in \mathbb{R}^n : \sum_{i \in [n]} \beta_i = 1 \textrm{ and } \forall i \in [n],\ \beta_i \geq 0\bigg\}.\]
And given a convex weight $\beta \in \Delta^{n-1}$, we let $f_\beta$ denote the \emph{scalarization}:
\begin{equation} \label{eqn:scalarization}
    f_\beta(x) := \sum_{i \in [n]} \beta_i f_i(x).
\end{equation}
For a detailed glossary, see \Cref{sec:notation}.

Let us recall the definition of a Pareto optimal decision vector.
\begin{definition}[Pareto optimality] \label{def:pareto-optimality}
Given objectives $f_1,\ldots, f_n$, we say that a decision vector $x \in \mathbb{R}^d$ is \emph{Pareto optimal} if for all $x' \in \mathbb{R}^d$:
\[\qquad f_i(x') < f_i(x) \qquad \implies \qquad \exists j \quad \mathrm{s.t.}\quad  f_j(x') > f_j(x).\]
We call the set of Pareto optimal decision vectors the \emph{Pareto set} of $f_1,\ldots, f_n$, denoted $\mathrm{Pareto}(F)$.
\end{definition}

In words, the above condition states that there is no way to improve $f_i$ without also worsening some other $f_j$. When the objectives are smooth, a related local condition is Pareto stationarity:

\begin{definition}[Pareto stationarity] \label{def:pareto-stationarity}
Given objectives $f_1,\ldots, f_n$, we say that a decision vector $x \in \mathbb{R}^d$ is \emph{Pareto stationary} if zero is a convex combination of the gradients:
\[\phantom{\qquad \textrm{for some }} \nabla f_\beta(x) = 0, \qquad \textrm{for some } \beta \in \Delta^{n-1},\]
where $f_\beta$ is defined by \Cref{eqn:scalarization}.
\end{definition}

In particular, Pareto stationarity is a necessary condition for Pareto optimality \citep{marucsciac1982fritz}. Furthermore, it is sufficient when the objectives are twice-differentiable and are strictly convex \citep{fliege2009newton}. As we have assumed this, we have:
\[x \in \mathrm{Pareto}(F) \quad\Longleftrightarrow\quad x \textrm{ is Pareto stationary}.\]

\section{The Pareto manifold}
It is not immediately evident from the definition of Pareto stationarity that $\mathrm{Pareto}(F)$ is amenable to the Pareto-constrained optimization problem defined in \Cref{eqn:pso}. In general, the Pareto set is non-smooth and non-convex. Even when the objectives are positive-definite quadratics, the Pareto set can have singularities \citep{sheftel2013geometry}. For example, see the Pareto set in \Cref{fig:hard-pareto}. 

This issue of non-smoothness arises because the set of Pareto stationary points naturally lives in a higher-dimensional space $\mathbb{R}^d \times \Delta^{n-1}$, in which it is a smooth $(n-1)$-dimensional submanifold. But when it is projected back down into $\mathbb{R}^d$, it can cross itself to create singularities. Formally, we define:

\begin{definition}[Pareto manifold] \label{def:pareto-manifold}
The \emph{Pareto manifold} $\mathcal{P}(F) \subset \mathbb{R}^d \times \Delta^{n-1}$ is the zero set:
\[\mathcal{P}(F) = \big\{(x,\beta) : \nabla f_\beta(x) = 0\big\}.\]
\end{definition}

The Pareto manifold consists of all $(x,\beta)$ such that $x$ is Pareto stationary and $\beta$ bears witness to the stationarity condition $\nabla f_\beta(x) = 0$. And of course, we can recover the Pareto set from the Pareto manifold simply by projecting down to its first component in $\mathbb{R}^d$:
\[\phantom{\textrm{ for some }\beta \in \Delta^{n-1}}x \in \mathrm{Pareto}(F) \quad \Longleftrightarrow \quad (x,\beta) \in \mathcal{P}(F) \textrm{ for some }\beta \in \Delta^{n-1}.\]
But this projection can also collapse any smoothness structure that $\mathcal{P}(F)$ has. And indeed, it is a smooth submanifold of $ \mathbb{R}^d \times \Delta^{n-1}$. To see this, notice that $\mathcal{P}(F)$ is the zero set of the map: 
\[(x,\beta) \mapsto \nabla f_\beta(x),\] 
whose partial derivative with respect to $x$ is invertible---the partial derivative is $\nabla^2 f_\beta$, which is positive-definite by strong convexity. The implicit function theorem then yields its manifold structure:

\begin{restatable}[Characterization of the Pareto manifold] {proposition}{paretochar}\label{prop:lifted-pareto-set}
    Define the map $x^* : \Delta^{n-1} \to \mathrm{Pareto}(F)$:
    \begin{equation} \label{eqn:x-best-response}
        x^*(\beta)\equiv x_\beta := \argmin_{x \in \mathbb{R}^d}\, f_\beta(x).
    \end{equation}
    Let $\nabla F(x) \in \mathbb{R}^{n \times d}$ be the Jacobian. Then, the map $x^*$ has derivative:
    \begin{equation} \label{eqn:x-best-response-gradient}
    \nabla x^*(\beta) = -\nabla^2f_\beta(x_\beta)^{-1} \nabla F(x_\beta)^\top,
    \end{equation}
    so that the map $\beta \mapsto (x_\beta,\beta)$ is a diffeomorphism of $\Delta^{n-1}$ with the Pareto manifold $\mathcal{P}(F)$.
\end{restatable}

Thus, one natural set of coordinates for the Pareto manifold is its parametrization by the simplex. This allows us to define an equivalent but smooth formulation of the Pareto-constrained optimization problem obtained by pulling $f_0$ back onto $\Delta^{n-1}$, which we shall now do.

\section{The Pareto-constrained optimization problem} \label{sec:pcop}
The Pareto-constrained optimization problem defined in \Cref{eqn:pso} has another formulation:
\begin{equation} \label{eqn:lifted-preference-optimization}
    \qquad \minimize_{(x,\beta) \in \mathcal{P}(F)}\ f_0(x),
\end{equation}
where the constraint has been replaced with the Pareto manifold. The two are equivalent because $\mathrm{Pareto}(F)$ is precisely the projection of $\mathcal{P}(F)$ onto $\mathbb{R}^d$. But the reformulation allows us to apply \Cref{prop:lifted-pareto-set} to pullback the optimization problem onto $\Delta^{n-1}$:
\begin{equation} \label{eqn:pullback-preference-optimization}
    \qquad \minimize_{\beta \in \Delta^{n-1}}\ (f_0 \circ x^*)(\beta),
\end{equation}
which is now a smooth optimization problem over the simplex. We say that $x$ is \emph{preference optimal} if it is a solution to (\ref{eqn:lifted-preference-optimization}); if $\beta$ solves (\ref{eqn:pullback-preference-optimization}), then correspondingly, $x^*(\beta)$ is preference optimal. 

As $f_0$ and $x^*$ are smooth, so too is their composition $(f_0 \circ x^*)$; we can define a stationarity condition in the standard way for smooth objectives on convex sets \citep{nesterov2003introductory}.  We say that $x$ is \emph{weakly preference stationary} if there is some $\beta$ such that $(x,\beta) \in \mathcal{P}(F)$ and $\beta$ is stationary in the usual sense for (\ref{eqn:pullback-preference-optimization}). For any given $x$, many $\beta$'s could satisfy the condition $(x,\beta) \in \mathcal{P}(F)$,
\begin{equation} \label{eqn:delta(x)}
    \Delta^{n-1}(x) := \big\{\beta \in \Delta^{n-1} : \nabla f_\beta(x) = 0\big\}.
\end{equation}
We say that $x$ is \emph{preference stationary} if the stationary condition holds for all such $(x,\beta)$'s.

\begin{restatable}[Preference stationarity]{definition}{prefstat} \label{def:preference-stationarity}
    We say that a point $x \in \mathrm{Pareto}(F)$ is \emph{weakly preference stationary} if there exists some $\beta \in \Delta^{n-1}(x)$ such that:\footnote{As $\nabla F(x_\beta)^\top \beta = \nabla f_\beta(x_\beta) = 0$, \Cref{eqn:lifted-pref-stationarity} can be simplified to $-\nabla (f_0\circ x^*)(\beta)^\top \beta' \leq 0$, for all $\beta' \in \Delta^{n-1}$.}
    \begin{equation} \label{eqn:lifted-pref-stationarity}
    \phantom{ \qquad \forall v \in T_{\Delta^{n-1}}(\beta)}-\nabla (f_0\circ x^*)(\beta)^\top (\beta' - \beta) \leq 0, \qquad \forall \beta' \in \Delta^{n-1},
    \end{equation}
    where \Cref{eqn:x-best-response-gradient} gives $\nabla x^*$. If (\ref{eqn:lifted-pref-stationarity}) holds for all $\beta \in \Delta^{n-1}(x)$, then $x$ is \emph{preference stationary}.
\end{restatable}

From optimization on convex sets \cite[Lemma 3.1.19]{nesterov2003introductory}, we immediately have:

\begin{proposition}[Necessary condition]
    Preference optimality implies (weak) preference stationarity.
\end{proposition}

\begin{figure}[t]
    \centering
    \captionsetup{width=0.9\textwidth}
    \input{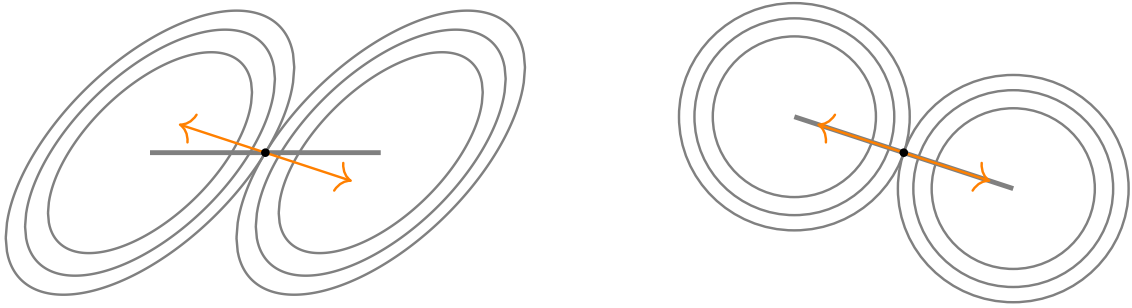}
    \caption{Two instances of $\mathrm{Pareto}(f_1, f_2)$ are shown (thick gray lines), where $f_1$ and $f_2$ are positive-definite quadratic objectives in $\mathbb{R}^2$ (visualized by contour lines). At $x$ (the black dot), the two instances  share the same local information $-\nabla f_1(x)$ and $- \nabla f_2(x)$ (orange arrows); they cross the contour lines at right angles. When $n = 2$, the Pareto set contains all $z$ such that $\nabla f_1(z) = -\lambda \nabla f_2(z)$ for $\lambda \geq 0$. Notice that if $f_0$ is strictly convex and $x$ does not minimize $f_0$ over $\mathbb{R}^2$, then $x$ cannot be stationary for both instances.}
    \label{fig:quadratics}
\end{figure}

While this definition of preference stationarity is appealing because it is necessary for preference optimality and because it is well-founded in standard optimization theory, it is not necessarily the only reasonable relaxation of preference optimality. For example, our notion of preference stationarity requires second-order information in $F$ for the term $\nabla^2 f_\beta(x)^{-1}$. It is natural to ask whether we could define stationarity with reference to only first-order information. It turns out that this is impossible, if we require the stationarity condition to be (i) non-trivial, (ii) necessary for preference optimality and (iii) decidable from local information at a single point $x$.

The reason is that the local behavior of the Pareto set about a point $x$ cannot be determined from $\nabla F(x)$ alone. \Cref{fig:quadratics} shows two different Pareto sets that share the same first-order information at a point $x$. But the preference stationarity of $x$ with respect to $f_0$ also depends on its neighboring Pareto points. So to attain a non-trivial and necessary condition, we would either need to look at higher-order information or more than a one point. To formalize this, first define:

\begin{definition}[Preference genericity]
    Let $\{v_0, v_1,\ldots, v_n\} \subset \mathbb{R}^d$ where $1 < n \leq d$. We say that this set is \emph{preference generic} if there is a unique $\beta \in \Delta^{n-1}$ such that $\beta_1 v_1 + \dotsm \beta_n v_n = 0$, and:
    \[v_0 \notin \mathrm{span}(v_1,\ldots, v_n).\]
\end{definition}

We also formalize stationarity conditions as \emph{decision functions}, which are functions mapping continuous inputs to Boolean variables taking values of \texttt{true} or \texttt{false}.

\begin{definition}[Stationarity function]
    A \emph{first-order stationary condition} is a decision function:
\[\mathrm{Stationary} : \mathbb{R}^d \overset{n + 1 \textrm{ times}}{ \times \ \ \dotsm\ \ \times }  \mathbb{R}^d \to \{\mathsf{true}, \mathsf{false}\}.\]
Let $f_0$ be a smooth preference function and $f_1,\ldots, f_n$ be smooth, strongly convex objectives. We say that a first-order condition is \emph{necessary} if the following holds:
\[x \textrm{ is preference optimal} \quad\implies \quad \mathrm{Stationary}\big(\nabla f_0(x), \ldots, \nabla f_n(x)\big) = \mathsf{true}.\]
\end{definition}

\begin{restatable}[Necessary first-order conditions are trivial]{proposition}{firstorder}\label{prop:impossibility} 
Suppose that $\mathrm{Stationary}$ is necessary. Then, it is trivial in the following sense: for any preference generic set of $v_0,\ldots, v_n \in \mathbb{R}^d$, 
\[\mathrm{Stationary}(v_0,\ldots, v_n) = \mathsf{true}.\]
\end{restatable}

\section{Pareto majorization-minimization}
Let us now consider how to solve the Pareto-constrained optimization problem:
\begin{equation} 
\minimize_{\beta \in \Delta^{n-1}}\ (f_0 \circ x^*)(\beta).\tag{\ref{eqn:pullback-preference-optimization}}
\end{equation}
As the problem has been reformulated as a smooth optimization problem on the simplex, this seems to open up local methods like gradient descent. But for this, there is a remaining issue that $x^*$ is defined implicitly as the solution of another optimization problem:
\begin{equation}
    x^*(\beta)\equiv x_\beta := \argmin_{x \in \mathbb{R}^d}\, f_\beta(x).\tag{\ref{eqn:x-best-response}}
\end{equation}
Because $x_\beta$ does not generally have a closed form, we also cannot explicitly compute $\nabla x^*(\beta)$, which is required if we wish to compute $\nabla (f_0 \circ x^*)(\beta)$ by the chain rule.

\subsection{Approximating the gradient}
We can, however, approximate the gradient. Define the following estimator, which uses local information $\nabla^2 f_\beta(x)$ and $\nabla F(x)$ at $x$ as a proxy for the corresponding local information at $x_\beta$:
\begin{equation} \label{eqn:dx-approx}
    \widehat{\nabla}x^*(x,\beta) := - \nabla^2 f_\beta(x)^{-1} \nabla F(x)^\top.
\end{equation}
If $F$ has continuous second derivative, then $\widehat{\nabla} x^*(x,\beta)$ approaches $\nabla x^*(\beta)$ as $x$ goes to $x_\beta$; strong convexity implies that $\nabla^2 f_\beta$ has a continuous inverse. And so, there are many reasonable approaches to this problem: it is a smooth optimization problem on a convex set with approximate gradients. For example, we could use the gradient estimate to perform projected gradient descent on the simplex. 

Then, the questions at hand: (a) how valid is the approximation $\widehat{\nabla}x^*(x,\beta)$, and (b) how can an optimization procedure make use of that information? It is certainly not the case that the approximation computed at $(x, \beta)$ for some distant $x$ should be as equally valid as one computed near $(x_\beta, \beta)$. One way we can capture the validity of the estimator $\widehat{\nabla} x^*(x,\beta)$ is by using it to construct a majorizing surrogate function, which is a function that upper bounds $f_0 \circ x^*$:
\begin{definition}[Majorizing surrogate]
    A function $g : \Delta^{n-1} \to \mathbb{R}$ \emph{majorizes} $f_0 \circ x^*$ if:
    \begin{equation} \label{eqn:majorizing}
    f_0(x_{\beta'}) \leq g(\beta'),
\end{equation}
for all $\beta' \in \Delta^{n-1}$. We say that $g$ is a \emph{surrogate} of $f$.
\end{definition}

Intuitively, the better the approximation is, the tighter an upper bound we could provably attain. And as an example, suppose that we have $\widehat{\nabla}x^*(x_\beta, \beta)$, which in this case is exactly $\nabla x^*(\beta)$. And suppose that we knew that $f_0 \circ x^*$ were 1-Lipschitz smooth. Then, the standard quadratic upper bound for Lipschitz smooth functions \citep{nesterov2003introductory} yields a family of majorizing surrogates: 
\[g(\beta'; x_\beta,\beta) = f_0(x_\beta) + \nabla f_0(x_\beta)^\top \widehat{\nabla} x^*(x_\beta,\beta) (\beta' - \beta) + \frac{1}{2} \|\beta' - \beta\|^2.\]
This means that we could use $g$ to bound how much improvement in $f_0$ is made by any iterative optimization scheme that takes a step from $\beta$ to $\beta'$: we can think of $g(\beta';x,\beta)$ as extracting information from $\widehat{\nabla}x^*(x,\beta)$ to certify when an update $f_0(x_{\beta'})$ will improve upon $f_0(x_{\beta})$.

\begin{algorithm}[t]
\caption{Pareto majorization-minimization (PMM)} \label{alg:pareto-mm}
	\textbf{Input:} objectives $F \equiv (f_1,\ldots, f_n)$, preference function $f_0$, and black-box optimizer $\widehat{\argmin}$\\
    \textbf{Initialize:} $(\beta_0, x_0) \in \Delta^{n-1} \times \mathbb{R}^d$
	\begin{algorithmic}[1]
		\FOR{$k=1,\ldots, K$}
        \STATE Compute a majorizing surrogate $g_k(\beta) \equiv g(\beta; x_k,\beta_k)$ satisfying \Cref{eqn:majorizing}
		\STATE Compute approximate minimizers
        \[\displaystyle\beta_{k+1} \leftarrow \widehat{\argmin_{\beta \in \Delta^{n-1}}}\, g_k(\beta)\qquad \textrm{and}\qquad \displaystyle x_{k+1} \leftarrow \widehat{\argmin_{x \in \mathbb{R}^d}}\, f_{\beta_{k+1}}(x).\]
	\ENDFOR
    \RETURN $(\beta_{K+1}, x_{K+1})$
	\end{algorithmic}
\end{algorithm}

\subsection{Algorithms from upper bounds}
Assuming we can obtain such bounds, we can use them not only to analyze optimization procedures, but we can also define a broad class of iterative methods that directly optimize the upper bounds. Suppose that we can compute a family of majorizing surrogates indexed over $\mathbb{R}^d \times \Delta^{n-1}$. Then, the idealized \emph{Pareto majorization-minimization} (PMM) algorithm proceeds in rounds:
\begin{enumerate}
    \item majorization: query $\widehat{\nabla}x^*(x_k,\beta_k)$ to construct a majorizing surrogate $g_k(\beta) \equiv g(\beta; x_k,\beta_k)$,
    \item minimization: make updates $\displaystyle \beta_{k+1} \leftarrow \argmin_{\Delta^{n-1}}\, g_k(\beta)$ and $\displaystyle x_{k+1} \leftarrow \argmin_{x \in \mathbb{R}^d}\,f_{\beta_{k+1}}$.
\end{enumerate}
The majorizing property of $g_k$ ensures that the iterates $f_0(x_{\beta_{k+1}})$ improve as $\beta_{k+1}$ optimizes $g_k$. We also operationalize the intuition that $g(\,\cdot\,;x,\beta)$ becomes more informative as $x$ approaches $x_\beta$ by optimizing $f_{\beta_{k+1}}$. \Cref{alg:pareto-mm} is obtained by relaxing step 2, for we do not need to fully optimize $g_k$ and $f_{\beta_{k+1}}$, and we allow for any black-box optimizer. In theory, any iterative optimization method could be interpreted as an approximate PMM; this yields one framework for convergence analysis.

\section{Approximability from smoothness}  \label{sec:assumptions}
In this section, we quantify the smoothness assumptions presented in \Cref{sec:preliminaries}. From them, we can derive the following implications:
\begin{itemize}
    \item \Cref{ass:objectives} allows us to bound the size of the Pareto set (\Cref{lem:diam}). 
    \item \Cref{ass:hesssmooth} additionally bounds the curvature of the Pareto manifold: we show that $\nabla x^*$ is well-behaved (\Cref{lem:x*-lipschitz}) and is well-approximated by $\widehat{\nabla}x^*$ (\Cref{lem:approx-dx-bound}). 
    \item \Cref{ass:f0} further leads to error bounds when approximating gradient of $f_0 \circ x^*$ (\Cref{lem:appro-gradient-bound}). It also allows us to define a notion of approximate preference stationarity that is geometrically meaningful (\Cref{prop:approx-stationary}) and can be verified using approximate information (\Cref{lem:approx-pref-stationarity-computable}). 
\end{itemize}
Formally, we have:

\begin{assumption} \label{ass:objectives}
Let the objectives $f_1,\ldots, f_n : \mathbb{R}^d \to \mathbb{R}$ be twice differentiable, $\mu$-strongly convex, and have $L$-Lipschitz continuous gradient. That is, for all $i = 1,\ldots, n$,
\begin{align*}
   \mu \mathbf{I} \preceq  \nabla^2f_i(x) \preceq L\mathbf{I}.
\end{align*}
Thus, the condition number of $\nabla^2 f_i$ is upper bounded by $\kappa := L / \mu$. We also let $r$ be a scale parameter, defined by the maximum distance between any of the minimizers of the objectives:
\[r := \max_{i,j \in [n]}\, \big\|\argmin f_i(x) - \argmin f_j(x)\big\|_2.\]
\end{assumption}

\begin{restatable}[Size of Pareto set]{lemma}{lemdiam} \label{lem:diam}
    Suppose $F$ satisfies \Cref{ass:objectives}. Then $R \leq \sqrt{\kappa} r$, where:
    \[R := \mathrm{diam}\big(\mathrm{Pareto}(F)\big) \equiv \sup \big\{ \|x - x'\|_2: x, x' \in \mathrm{Pareto}(F)\big\}.\]
\end{restatable}

\begin{assumption} \label{ass:hesssmooth}
Let the objectives $f_1,\ldots, f_n : \mathbb{R}^d \to \mathbb{R}$ have $L_H$-Lipschitz continuous Hessian. That is, for all $x, y \in\mathbb{R}^d$ and $i = 1,\ldots, n$, we have $\norm{\nabla^2 f_i(x)-\nabla^2 f_i(y)}_2\leq L_{H}\norm{x-y}_2$.
\end{assumption}

\begin{restatable}[Smoothness of $x^*$]{lemma}{xlipschitz} \label{lem:x*-lipschitz}
    Suppose $F$ satisfies Assumptions~\ref{ass:objectives},\ref{ass:hesssmooth}. Then, $x^* : \Delta^{n-1} \to \mathbb{R}^d$ is $M_0$-Lipschitz continuous and has $M_1$-Lipschitz continuous gradients, where:
    \[M_0 := \kappa R\qquad \textrm{and}\qquad M_1:= 2\kappa^2R \left(1 + \frac{L_H R}{\mu}\right).\]
\end{restatable}

\begin{restatable}[Approximability of $\nabla x^*$]{lemma}{approxdx} \label{lem:approx-dx-bound}
    If $F$ satisfies Assumptions~\ref{ass:objectives},\ref{ass:hesssmooth}. Then:
    \[\|\nabla x^*(\beta) - \widehat{\nabla} x^*(x,\beta)\|_{1,2} \leq \frac{1}{\mu} \frac{M_1}{2M_0} \|\nabla f_\beta(x)\|_2.\]
\end{restatable}

\begin{assumption}\label{ass:f0}
    Let the preference function $f_0 : \mathbb{R}^d \to \mathbb{R}$ have $L_0$-Lipschitz continuous gradient. That is, for all $x, y \in \mathbb{R}^d$, we have $\norm{\nabla f_0(x)-\nabla f_0(y)}_2\leq L_{0}\norm{x-y}_2$. 
\end{assumption}

\begin{restatable}[Approximability of $\nabla(f_0 \circ x^*)$]{lemma}{approxgradient} \label{lem:appro-gradient-bound}
    If $F$ and $f_0$ satisfy Assumptions~\ref{ass:objectives},\ref{ass:hesssmooth},\ref{ass:f0}. Then:
    \[\big\|\nabla (f_0 \circ x^*)(\beta)^\top - \nabla f_0(x)^\top \widehat{\nabla} x^*(x,\beta)\big\|_{1,2} \leq \frac{1}{\mu} \left(\frac{M_1}{2M_0}\|\nabla f_0(x)\|_2 + L_0M_0 \right) \|\nabla f_\beta(x)\|_2.\]
    We denote the right-hand side by $\mathrm{err}_{\nabla f_0}(x,\beta)$.
\end{restatable}

\subsection{An approximate solution concept}
In practice, we generally can never exactly recover stationary points, so we further relax our target solution concept to an approximate version of preference stationarity in the standard way \citep{nesterov2013gradient}. To define our notion of approximation, we consider $\Delta^{n-1}$ as a metric space. While somewhat arbitrary, it is also fairly natural to endow $\Delta^{n-1}$ with the $\ell_1$-metric, so that it has unit diameter.

\begin{restatable}[Approximate preference stationarity]{definition}{approxprefstat} \label{def:approx-pref-stationarity}
    Let $\epsilon_0, \epsilon \geq 0$. A point $(x,\beta) \in \mathbb{R}^d \times \Delta^{n-1}$ is \emph{$(\epsilon_0, \epsilon)$-preference stationary} if:
    \begin{subequations}\label{eqn:approx-stationarity}
    \begin{alignat}{1}
     - \nabla f_0(x_\beta)^\top \nabla x^*(\beta) (\beta' - \beta) &\leq \epsilon_0 \|\beta' - \beta\|_1,\qquad \forall \beta' \in \Delta^{n-1}. \tag{\ref*{eqn:approx-stationarity}a}\\
    \|\nabla f_\beta(x)\|_2 &\leq \epsilon
    \tag{\ref*{eqn:approx-stationarity}b}
    \end{alignat}
    \end{subequations}
\end{restatable}

When the objectives and preference are sufficiently nice, then an approximate preference stationary solution $(\hat{x},\hat{\beta})$ has an intuitive meaning: (a) there is a ball around $\hat{\beta}$ within which $f_0 \circ x^*$ decreases at most at an $O(\epsilon_0)$-rate when moving away from $\hat{\beta}$, and (b) the point $\hat{x}$ is $O(\epsilon)$-close to $x_{\hat{\beta}}$.

\begin{restatable}[Geometric meaning of approximate stationarity]{proposition}{propapproxstat} \label{prop:approx-stationary}  
    Let $F$ and $f_0$ satisfy Assumptions~\ref{ass:objectives},\ref{ass:hesssmooth},\ref{ass:f0} and let $(\hat{x}, \hat{\beta})$ be $(\epsilon_0, \epsilon)$-preference stationary. The following hold:
    \begin{enumerate}
        \item[a.] if $\|\beta - \hat{\beta}\|_1 \leq s$, then $\displaystyle f_0(x_{\beta}) - f_0(x_{\hat{\beta}})\geq  - 2\epsilon_0  \|\beta - \hat{\beta}\|_1$, and
        \item[b.] $\|\hat{x} - x_{\hat{\beta}}\|_2 \leq \epsilon/\mu$,
    \end{enumerate}
    where we let $R$ is defined in \Cref{lem:diam} and  $s := \frac{2 \mu^2 \epsilon_0}{L_0 L^2 R^2}$. 
\end{restatable}

\begin{restatable}[Verifiability of approximate stationarity]{lemma}{approxpref} \label{lem:approx-pref-stationarity-computable}
    Let $F$ and $f_0$ satisfy Assumptions~\ref{ass:objectives},\ref{ass:hesssmooth},\ref{ass:f0}. Then $(\hat{x},\hat{\beta})$ is $(\epsilon_0, \epsilon)$-preference stationary if $\|\nabla f_{\hat{\beta}}(\hat{x})\|_2 \leq \epsilon$, and for some $x \in \mathbb{R}^d$ and $\alpha \in (0,1)$,
    \begin{enumerate}
        \item an $\alpha \cdot \epsilon_0$-approximate stationary condition holds:
        \begin{equation} \label{eqn:approx-stationary-condition}
        -\nabla f_0(x)^\top \widehat{\nabla} x^*(x,\hat{\beta}) (\beta' - \hat{\beta}) \leq \alpha \cdot \epsilon_0 \|\beta' - \hat{\beta}\|_1,
        \end{equation}
        \item an error bound holds:
        \[\mathrm{err}_{\nabla f_0}(\hat{\beta}, x) \leq (1 - \alpha) \cdot \epsilon_0.\]
    \end{enumerate}
\end{restatable}

\section{Analysis of Pareto majorization-minimization}
\label{sec:analysis}
In this section, we give an explicit majorizing family of positive-definite quadratics surrogates, and we provide a condition for when \Cref{alg:pareto-mm} converges.

\subsection{A family of majorizing surrogates}
Because $f_0$ and $x^*$ are respectively $L_0$- and $M_1$-Lipschitz smooth (\Cref{ass:f0} and \Cref{lem:approx-dx-bound}), their composition is also Lipschitz smooth and admits the quadratic upper bound:
\begin{align*} 
f_0(x_{\beta'}) &\leq f_0(x_\beta) + \nabla (f_0 \,\circ\, x^*)(\beta)^\top (\beta' - \beta) + \frac{1}{2} n L_0 M_1 \|\beta' - \beta\|_2^2,
\end{align*}
where the dimension $n$ in the second term comes from $\|\beta' - \beta\|_1^2 \leq n \|\beta' - \beta|_2^2$. And even though the gradient $\nabla (f_0 \circ x^*)(\beta)^\top$ is implicit, we can approximate it using $\nabla f_0(x)^\top \widehat{\nabla}x^*(x,\beta)$ where the error is bounded by \Cref{lem:appro-gradient-bound}. This implies the following family of majorizing surrogates:
\begin{restatable}[A family of majorizing surrogates]{proposition}{majsur} \label{prop:maj-sur}
Suppose $F$ and $f_0$ satisfy Assumptions~\ref{ass:objectives},\ref{ass:hesssmooth},\ref{ass:f0}. Let $\mathrm{err}_{\nabla f_0}(x,\beta)$ be as defined in \Cref{lem:appro-gradient-bound}. Define:
    \begin{equation} \label{eqn:surrogate}
    g (\beta'; x, \beta) := f_0(x_\beta) + \nabla f_0(x)^\top \widehat{\nabla} x^*(x,\beta) (\beta' - \beta) + \frac{1}{2} \mu_g \|\beta'- \beta\|_2^2 + \mathrm{err}_{\nabla f_0}(x,\beta),
\end{equation} 
where $\mu_g := n L_0 M_1$. Then $g(\beta'; \beta,x)$ majorizes $f_0 \circ x^*$, satisfying \Cref{eqn:majorizing}.
\end{restatable}

Note that technically we cannot explicitly compute the value $g(\beta'; x,\beta)$ because it contains the term $f_0(x_\beta)$. However, we can compute the difference $g(\beta'; x,\beta) - g(\beta; x, \beta)$, which is enough to optimize $g$ and to prove descent for the iterates of any given optimization scheme.

\subsection{Convergence analysis}

We now give the convergence result for the Pareto majorization-minimization algorithm. We make use of the sufficient condition provided by \Cref{lem:approx-pref-stationarity-computable}, which can be determined using approximate information. And as \Cref{alg:pareto-mm} can make use of any black-box optimizer, we state the result in terms of the convergence guarantees of the black-box optimizers. 

In particular, the PMM algorithm uses two optimizers: one for the surrogate $g(\,\cdot\,;x, \beta)$ and another for the scalarized objective $f_\beta(\cdot)$. As we aim to achieve $\epsilon_0$-preference stationarity, we also ask the optimizer for the surrogate $g$ to achieve $O(\epsilon_0)$-approximate stationarity. 

But approximate stationarity with respect to $g$ only transfers to $f_0$ when the  surrogate is sufficiently tight, which depends on the performance of the optimizer for $f_\beta(\cdot)$. It turns out that we shall require that it achieves $O(\epsilon_0^2)$-optimality. This is because when we optimize a positive-definite quadratic over a convex set, finding an $\epsilon_0$-approximate stationary point $\hat{\beta}$ means finding an $O(\epsilon_0^2)$-approximately optimal point (\Cref{lem:approx-stationary-implies-close,lem:Q-descent}):
\[g(\hat{\beta}; x, \beta) < g(\beta^*; x,\beta) + O(\epsilon_0^2),\]
where $\beta^*$ minimizes the surrogate. But, the surrogate contains an approximation error $\mathrm{err}_{\nabla f_0}(\beta,x)$. If this error term is larger than $\Omega(\epsilon_0^2)$, then it is possible for the surrogate to fail to either (i) decide that the current iterate $\beta$ is $\epsilon_0$-preference stationary or (ii) make progress by finding some $\hat{\beta}$ that certifiably improves on $f_0$. We preclude this by requiring the optimizer for $f_\beta$ to achieve $O(\epsilon_0^2)$-optimality.

\begin{restatable}[Convergence of PMM]{theorem}{pmmconvergence} \label{thm:pmm-convergence}
    Let $F$ and $f_0$ satisfy Assumptions~\ref{ass:objectives},\ref{ass:hesssmooth},\ref{ass:f0}. Fix $0 < \epsilon^{1/2} \leq \epsilon_0 \leq 1$. Let $\hat{x}_\beta$ and $\hat{\beta}$ be the approximate solutions that are returned by the black-box optimizer for $g(\,\cdot\, ;x,\beta)$ and $f_\beta(\cdot)$, defined in \Cref{eqn:surrogate} and \Cref{eqn:scalarization}, respectively:
    \[\hat{\beta} \leftarrow \widehat{\argmin_{\beta' \in \Delta^{n-1}}} \,g(\beta'; x,\beta) \qquad \textrm{and} \qquad \hat{x}_\beta \leftarrow \widehat{\argmin_{x \in \mathbb{R}^d}} \,f_\beta(x).\]
    Given constants $c_1, c_2 > 0$, suppose that the black-box optimizer achieves the following guarantees:
    \begin{enumerate}
        \item the approximate minimizer $\hat{\beta}$ is $O(\epsilon_0)$-approximately stationary:
            \[\phantom{\qquad\forall v \in T_{\Delta^{n-1}}(\beta)}- \nabla g(\hat{\beta}; x,\beta)v \leq c_1\cdot \epsilon_0 \|v\|_2,\qquad\forall v \in T_{\Delta^{n-1}}(\beta).\]
        \item the approximate minimizer $\hat{x}_\beta$ is an $O(\epsilon_0^2)$-approximate solution:
        \[\|\nabla f_\beta(\hat{x}_\beta)\| \leq c_2 \cdot \epsilon.\]
    \end{enumerate}
    Let $(x_k,\beta_k)_k$ be the iterates of \Cref{alg:pareto-mm}. Then, there exist $c_1(f_0, F)$ and $ c_2(f_0, F)$ bounded away from zero and some $K$ such that $(f_0 \circ x^*)(\beta_k)$ is monotonically decreasing for $k \in [K]$ and $(x_K,\beta_K)$ is an $(\epsilon_0, \epsilon)$-preference stationary point. Furthermore, $K$ is no more than $O(\epsilon_0^{-2})$:
    \[K \leq \frac{2 \mu_g\cdot \big(f^* - f_*\big)}{ c_1^2 \cdot \epsilon_0^2},\]
    where $f^* := \max f_0(x)$ and $ f_* = \min f_0(x)$ are optimized over the compact set $\mathrm{Pareto}(F)$.
\end{restatable}
\begin{remark}
    Algorithm~\ref{alg:pareto-mm} makes calls to sub-routines at each iteration to solve two sub-problems. As the problems are strongly-convex and Lipschitz-smooth, they can be solved using (projected) gradient descent with iteration complexity $O(\log(1/\epsilon_0))$. And so, taking the computational cost of the sub-problems into account only increases the rate obtained in Theorem~\ref{thm:pmm-convergence} by logarithmic factors.
\end{remark}

\section{Conclusion}
In this work, we provide a principled and efficient way to select a decision vector from the Pareto set of a set of objectives $f_1,\ldots, f_n$ given an additional preference function $f_0$. A main contribution of this work is to provide a geometrically-meaningful notion of (approximate) preference stationarity. This is non-trivial due to the non-smoothness and non-convexity of the Pareto set. We also provide a simple algorithm that achieves $\epsilon_0$-approximate stationarity with iteration complexity of $O(\epsilon_0^{-2})$.

%% file: appendix.tex
\section{Proofs and derivations}\label{sec:notation}
\begin{table}[h!]
    \centering
    \begin{tabular}{c|lm{20em}}
        \textbf{Symbol} & \textbf{Usage} \\ \hline 
        $\Delta^{n-1}$ & the $(n-1)$-simplex equipped with the $\ell_1$-metric, see \Cref{def:pareto-stationarity}\\
        $\Delta^{n-1}(x)$ & the set of $\beta$ satisfying $\nabla f_\beta(x) = 0$, see \Cref{eqn:delta(x)}\\
        $\nabla x^*$, $\widehat{\nabla} x^*$ & derivative of the map $x^*$ and its approximation, see \Cref{eqn:x-best-response-gradient,eqn:dx-approx} \\
        $\mathrm{err}_{\nabla f_0}(x,\beta)$ & bound on the approximation error of $\nabla (f_0 \circ x^*)$, see \Cref{lem:appro-gradient-bound} \\
        $F$, $(f_1,\ldots, f_n)$ & the set of objective functions\\
        $f_0$ & the preference function\\
        $f_\beta(x)$ & the scalarized objective $\sum_i \beta_i f_i(x)$, see \Cref{eqn:scalarization}\\
        $g(\beta'; x,\beta)$ & majorizing surrogate for $f(x_{\beta'})$, see \Cref{eqn:majorizing}\\
        $\kappa$ & condition number $\kappa := L/\mu$ for $\nabla^2 f_i$, see \Cref{ass:objectives}\\
        $L$, $L_H$, $L_0$ & Lipschitz parameters for $\nabla f_i$, $\nabla ^2 f_i$, and $\nabla f_0$,  see Assumptions~\ref{ass:objectives}, \ref{ass:hesssmooth}, \ref{ass:f0}\\
        $M_0$, $M_1$ & Lipschitz parameters for $x^*$ and $\nabla x^*$, see \Cref{lem:x*-lipschitz}\\
        $\mu$ & strong convexity parameter for $f_i$, see \Cref{ass:objectives} \\
        $\mu_g$ & strong convexity parameter $n L_0M_1$ for the surrogate $g$, see \Cref{eqn:majorizing}\\
        $\mathrm{Pareto}(F)$ & the set of Pareto optimal solutions of $F$, see \Cref{def:pareto-optimality}\\
        $r$ & distance between the minimizers of $f_1,\ldots, f_n$, see \Cref{ass:objectives}\\
        $\mathcal{P}(F)$ & the Pareto manifold, see \Cref{def:pareto-manifold}\\
        $R$ & $\mathrm{diam}\big(\mathrm{Pareto}(F)\big) := \sup \big\{ \|x - x'\|_2: x, x' \in \mathrm{Pareto}(F)\big\}$, see \Cref{lem:diam} \\
        $x^*(\beta)$, $x_\beta$ & stationary point for $f_\beta$, see \Cref{eqn:x-best-response}
    \end{tabular}
    \label{tab:notation}
\end{table}

\subsection{The Pareto manifold}
\paretochar*

\paragraph{Proof of \Cref{prop:lifted-pareto-set}} The map $x^*$ is well-defined because $f_\beta$ is strictly convex---it is the convex combination of strictly convex objectives, so it has a unique minimizer. Furthermore, because the objectives are smooth, the stationarity condition $\nabla f_\beta(x) = 0$ uniquely holds at $x^*(\beta)$:
\[\nabla f_\beta(x_\beta) = 0.\]
Define the map $\zeta(x,\beta) = \nabla f_\beta(x)$. Then, the Pareto manifold is precisely the zero set $\mathcal{P}(F) = \zeta^{-1}(0)$, and which can be parametrized by simplex $\Delta^{n-1}$ via the map $\beta \mapsto (x_\beta, \beta)$.

In fact, it is a smooth parametrization. To see this, we apply the implicit function theorem (\Cref{thm:implicit-function}), which states that the map $x^*$ is smooth at $\beta$ when $\nabla_x\zeta(x_\beta, \beta)$ is invertible. Indeed, we have that $\zeta$ is continuously differentiable, with:
\begin{align*}
    \nabla_x\zeta(x,\beta) &= \sum_{i \in [n]} \beta_i \nabla^2 f_i(x) = \nabla^2 f_\beta(x),\\
    \nabla_\beta \zeta(x,\beta) &= \nabla_\beta\left(\sum_{i \in [n]}\beta_i \nabla f_i(x)\right) = \nabla F(x)^\top.
\end{align*}
Because $f_\beta$ is strictly convex, it has positive definite Hessian, implying invertibility $\mathrm{det} \nabla_x \zeta(x_\beta, \beta) \ne 0$. Furthermore, \Cref{thm:implicit-function} also implies that the derivative of $\nabla x^*$ is given by \Cref{eqn:x-best-response-gradient}. It follows that the map $\beta \mapsto (x_\beta, \beta)$ is smooth. It also has a smooth inverse. Namely, the projection onto the second component $(x_\beta, \beta) \mapsto \beta$. Thus, $\mathcal{P}(F)$ is diffeomorphic with $\Delta^{n-1}$.\hfill $\blacksquare$

\vspace{11pt}

\begin{theorem}[Implicit function theorem, \cite{spivak2018calculus}] \label{thm:implicit-function}
    Let $f : \mathbb{R}^d \times \mathbb{R}^n \to \mathbb{R}^d$  be continuously differentiable on an open set containing $(a,b)$ and let $f(a,b) = 0$. Let $\nabla_u f(u,v)$ be the $d \times d$ matrix:
    \[\big[\nabla_u f(u,v)\big]_{ij} = \nabla_{u_j} f_i(u,v).\]
    If $\mathrm{det}\, \nabla_u f(a,b) \ne 0$, there are open sets $U \subset \mathbb{R}^d$ and $V \subset \mathbb{R}^n$ containing $a$ and $b$ respectively with the following property: for each $v \in V$ there is a unique $g(v) \in U$ such that $f(g(v),v) = 0$. Furthermore, the map $g$ is differentiable with derivative given by:
    \[\nabla g(v) = - \big[\nabla_u f(g(v),v) \big]^{-1} \nabla_vf(g(v),v).\]
\end{theorem}

\subsection{Solution concepts to Pareto-constrained optimization}

In this section, we elaborate on how the different solution concepts (optimality, stationarity, approximate stationarity) relate to each other for the Pareto-constrained optimization problem:
\begin{equation} 
    \qquad \minimize_{\beta \in \Delta^{n-1}}\ (f_0 \circ x^*)(\beta). \tag{\ref{eqn:pullback-preference-optimization}}
\end{equation}

We can call any optimal solution \emph{preference optimal}:
\begin{definition}[Preference optimality]
    A decision vector $x \in \mathrm{Pareto}(F)$ is \emph{preference optimal} if:
    \[\phantom{\textrm{for all $x' \in \mathrm{Pareto}(F)$}.}f_0(x) \leq f_0(x'),\qquad \textrm{for all $x' \in \mathrm{Pareto}(F)$}.\]
\end{definition}

Because preference optimality is a global condition, it is generally computationally infeasible to verify. By considering \Cref{eqn:pullback-preference-optimization} as a smooth optimization problem over the simplex, we relax the solution concept in the standard way to the first-order stationarity condition in terms of $\beta$:
\begin{equation}
    \phantom{\qquad \forall v \in T_{\Delta^{n-1}}(\beta).}- \nabla (f_0 \circ x^*)(\beta)(\beta' - \beta)  \leq 0, \qquad \forall \beta' \in \Delta^{n-1}.\tag{\ref{eqn:lifted-pref-stationarity}}
\end{equation}
Given a stationary point $\beta$, we can push forward this stationary condition to $x^*(\beta)$, which we say is \emph{weakly preference stationary}. We reproduce the definition from before:

\prefstat*

Finally, to relax the exact stationarity condition to an approximate one, we appeal to the standard notion of an approximate stationary point \citep{nesterov2013gradient}. In our setting, we can make use of:

\begin{definition}[Approximate stationary point, \cite{marumo2023majorization}] \label{def:general-approx-stationary}
    Let $\mathcal{C}$ be a closed and convex set, and let $f :\mathcal{C} \to \mathbb{R}$ be a smooth objective function. A point $\beta \in \mathcal{C}$ is an \emph{$\epsilon$-approximate stationary point} of $f$ if for all $\beta' \in \mathcal{C}$, the following holds:
    \[- \nabla f(\beta)^\top (\beta' - \beta) \leq \epsilon \|\beta' - \beta\|.\]
\end{definition}

Specializing \Cref{def:general-approx-stationary} to the optimization of $f_0 \circ x^*$ over $\Delta^{n-1}$ yields an approximate stationary condition for $\beta$. And because we are ultimately interested in $x^*(\beta)$, which is the solution of to optimizing $f_\beta$ over $\mathbb{R}^d$, we can also make use of \Cref{def:general-approx-stationary} to also define the appropriate approximate stationary condition on $x$. This leads us to \Cref{def:approx-pref-stationarity}, which we reproduce here:

\approxprefstat* 
\propapproxstat*

\paragraph{Proof of \Cref{prop:approx-stationary}}
    \begin{enumerate}
        \item[(a)] Recall that $x_\beta$ is the minimizer of $f_\beta$, by definition. Because $f_\beta$ is $\mu$-strongly convex, we can bound the distance between $x$ and $x_\beta$ by:
    \[\|x - x_\beta\| \leq \frac{1}{\mu} \|\nabla f_\beta(x)\| \leq \frac{\epsilon}{\mu},\]
    where the second inequality follows from condition (\hyperref[eqn:approx-stationarity]{\ref*{eqn:approx-stationarity}a}).
    \item[(b)] Let $\beta_s := (1-s) \beta + s \beta'$ parametrize the line connecting $\beta$ and $\beta'$. Let $\gamma : [0,1] \to \mathrm{Pareto}(F)$ be the path $\gamma(s) := x^*(\beta_s)$, so that:
    \begin{align*} 
        d\gamma(s)&= \nabla x^*(\beta_s) (\beta' - \beta)\, ds.
    \end{align*}
    We can now upper bound the difference:
    \begin{align*}
        f_0(x_\beta) - f_0(x_{\beta'}) &= - \int_\gamma \nabla f_0(x_{\beta_s})^\top d\gamma(s) 
        \\&= - \int_\gamma \big[\textcolor{blue}{\nabla f_0(x_{\beta_s}) - \nabla f_0(x_{\beta})}  + \textcolor{orange}{\nabla f_0(x_\beta)}\big]^\top d\gamma(s)
        \\&\leq \int_\gamma  \textcolor{blue}{L_0\|x_{\beta_s} - x_\beta\|} \, |d\gamma(s)| + \int_\gamma \big(-\textcolor{orange}{\nabla f_0 (x_\beta)}^\top d\gamma(s)\big).
    \end{align*}
    Let's bound the integrals separately. Since $x_{\beta_s} = \displaystyle\int_0^s d\gamma(s)(\beta' - \beta)$, we have by \Cref{lem:dx-bound}:
    \[\|x_{\beta_s} - x_\beta\| \leq \frac{LR}{\mu} \|\beta - \beta'\|_1\cdot s.\]
    We also have $|d \gamma(s)| \leq \mu^{-1} LR \|\beta - \beta'\|_1$, by \Cref{lem:dx-bound}. The first integral is bounded by:
    \begin{align*}
        \int_\gamma  \textcolor{blue}{L_0\|x_{\beta_s} - x_\beta\|} \, |d\gamma(s)| &\leq \int_0^1 \frac{L_0 L^2R^2}{\mu^2} \|\beta - \beta'\|_1^2 \cdot s \, ds = \frac{1}{2}  \frac{L_0 L^2 R^2}{\mu^2}\|\beta - \beta'\|_1^2.
    \end{align*}
    For the second integral, first note that condition (\hyperref[eqn:approx-stationarity]{\ref*{eqn:approx-stationarity}b}) implies:
    \[- \nabla f_0(x_\beta)^\top d\gamma(s) = - \nabla f_0(x_\beta)^\top \nabla x^*(x_\beta)(\beta' - \beta) \leq \epsilon_0 \big\|\beta - \beta'\big\|_1,\]
    yielding the other bound:
    \[\int_\gamma \big(-\textcolor{orange}{\nabla f_0 (x_\beta)}^\top d\gamma(s)\big) \leq \int_0^1 \epsilon_0  \|\beta - \beta'\|_1 \, ds = \epsilon_0 \|\beta - \beta'\|_1.\]
    Putting these two together, we obtain:
    \[f_0(x_\beta) - f_0(x_{\beta'}) \leq \frac{1}{2} \frac{L_0 L^2 R^2}{\mu^2}\|\beta - \beta'\|_1^2 + \epsilon_0 \|\beta - \beta'\|_1. \]
    It follows that if we restrict $\|\beta - \beta'\|_1 \leq \frac{2 \mu^2 \epsilon_0}{L_0 L^2 R^2}$, one of the factors of $\|\beta - \beta'\|_1$ in the first term can be absorbed into the constant, proving the result:
    \[f_0(x_\beta) \leq f_0(x_{\beta'}) + 2\epsilon_0 \|\beta - \beta'\|_1.\]
    \hfill $\blacksquare$
\end{enumerate}

\subsubsection{Weak preference stationarity and degeneracy} \label{sec:relationship}

The solution concepts are related by:
\[\substack{\textrm{preference}\\\textrm{optimality}} \quad \subset \quad
\substack{\textrm{preference}\\\textrm{stationarity}} \quad \subset \quad
\substack{\textrm{weak preference}\\\textrm{stationarity}} \quad \subset \quad
\substack{\textrm{approximate preference}\\\textrm{stationarity}}\]
It is fairly clear that the first and last inequalities are strict. Here, we discuss the inner inequality. 

It turns out that a point $x$ can be weakly preference stationary without being preference stationary. However, this can only happen if $x$ is also a point of singularity in $\mathrm{Pareto}(F)$. Geometrically, if we consider $\mathrm{Pareto}(F)$ as the projection of $\mathcal{P}(F)$ onto its first component in $\mathbb{R}^d$, the this means that multiple points are collapsed onto $x$. Algebraically, this means that the set of gradients $\nabla f_1(x),\ldots, \nabla f_n(x)$ fails to have full (Pareto) rank \citep{smale1973global,hamada2020topology}. 

To elaborate, recall the set:
\begin{equation} 
    \Delta^{n-1}(x) := \big\{\beta \in \Delta^{n-1} : \nabla f_\beta(x) = 0\big\}.\tag{\ref{eqn:delta(x)}}
\end{equation}
Then, $x$ is Pareto stationary if there is some $\beta$ in $\Delta^{n-1}(x)$, so that:
\[\sum_{i \in [n]} \beta_i \nabla f_i(x) = 0,\]
and the rank of this set of gradients is at most $n-1$. Since $\Delta^{n-1}$ does not contain any collinear vectors, if $\Delta^{n-1}(x)$ contains more than a single point, then the rank of the set of gradients must be strictly less than $n-1$. This leads us to the definition:

\begin{definition}[Pareto genericity]
    Let $\{v_1,\ldots, v_n \} \subset \mathbb{R}^d$. This set is \emph{Pareto generic} if:
    \[\phantom{\qquad\textrm{ for some }\beta \in \Delta^{n-1},}\beta_1 v_1 + \dotsm + \beta_n v_n = 0,\qquad\textrm{ for some }\beta \in \Delta^{n-1},\]
    and the non-degeneracy condition holds: $\mathrm{rank}(v_1,\ldots, v_n) = n-1$.
\end{definition}

If $\nabla F(x)$ is Pareto generic, then $\Delta^{n-1}(x)$ contains a unique $\beta$, so we immediately have:

\begin{proposition}[Generic and weak implies strong preference stationarity]
    If $\nabla F(x)$ is Pareto generic and $x$ is weakly preference stationary, then $x$ is preference stationary.
\end{proposition}

However, when the gradients $\nabla F(x)$ are not Pareto generic, then weak preference stationarity can be strictly weaker. Let $(x, \beta)$ where $\beta \in \Delta^{n-1}(x)$ be weakly preference stationary, so that:
\[\phantom{,\qquad \forall \beta' \in \Delta^{n-1}.}-\nabla f_0(x)^\top \underbrace{\big(-\nabla^2f_\beta(x)^{-1} \nabla F(x)^\top (\beta' - \beta) \big)}_{\nabla x^*(\beta) (\beta' - \beta)}\leq 0,\qquad \forall \beta' \in \Delta^{n-1}.\]
We can simplify this by using the fact that $\nabla f_\beta(x) = \nabla F(x)^\top \beta = 0$. Then, one way for the stationary condition to be fulfilled is for the underlined term to be normal to $\nabla f_0(x)$:
\[\phantom{,\qquad \forall \beta' \in \Delta^{n-1}.}-\nabla^2 f_\beta(x)^{-1} \nabla F(x)^\top \beta' \in \mathrm{span}\big(\nabla f_0(x)\big)^\perp,\qquad \forall \beta' \in \Delta^{n-1}.\]
This statement has the following geometric interpretation. These vectors are contained in the Clarke tangent cone of $\mathrm{Pareto}(F)$ at $x$. If these are the only vectors in the tangent cone, then this above condition states that $-\nabla f_0(x)$ is contained in the normal cone of $\mathrm{Pareto}(F)$ at $x$. 

But, in general, the tangent cone contains the union of subspaces:
\[\bigcup_{\beta \in \Delta^{n-1}(x)} \left\{-\nabla^2 f_\beta(x)^{-1} \nabla F(x)^\top \beta' : \beta' \in \Delta^{n-1}\right\}.\]
And so, when $\Delta^{n-1}(x)$ does not contain a unique vector, the tangent cone can contain more vectors. By selecting different $\beta$'s, we recover different slices of the tangent cone. This also means that even if the above normality condition holds for one $\beta$, it may fail to hold for a different $\tilde{\beta} \in \Delta^{n-1}(x)$. In this case, $(x,\beta)$ is weakly preference stationary while $(x,\tilde{\beta})$ may not be.

\subsubsection{Insufficiency of first-order information} \label{sec:impossiblity}

\firstorder*

\paragraph{Proof of \Cref{prop:impossibility}} 
It suffices to show that there exist $f_0$, $F$, and $x^\star$ such that $x^\star$ is preference optimal and for $i = 0,\ldots, n$:
\begin{equation} \label{eqn:matching-conditions}
v_i = \nabla f_i(x^\star).
\end{equation}
And since $x^\star$ is preference optimal, any necessary stationary condition must accept:
\[\mathrm{Stationary}(v_0,\ldots, v_n) = \mathsf{true}.\]

Without loss of generality, let $x^\star = 0$ by an affine transformation. To construct $f_0$ and $F$, we can simply consider a family of positive-definite quadratics:
\begin{itemize}
    \item Let the preference function $f_0$ be:
    \[f_0(x) = \frac{1}{2} \|x + v_0\|^2.\]
    Notice that $\nabla f_0(x^*) = v_0$.
    \item Let the objectives $f_1,\ldots, f_n$ share the same Hessian:
    \[f_i(x) = \frac{1}{2} \|A(x - z_i)\|^2,\]
    where $A \in \mathbb{R}^{d \times d}$ is full rank and $z_i \in \mathbb{R}^d$. Let $H = A^\top A$ for short.
\end{itemize}
We show that we can set $A$ and the $z_i$'s so that $x^\star$ is preference optimal while \Cref{eqn:matching-conditions} holds. 

By \Cref{lem:conv-hull}, the Pareto set is the convex hull $\mathcal{C} := \mathrm{conv}(z_1,\ldots, z_n)$. Notice that the choice of $H$ and $v_i$'s determines the $z_i$'s, since we require $\nabla f_i(x^*) = v_i$, which expands to:
\[\phantom{\qquad \forall i \in [n]}z_i = - H^{-1} v_i, \qquad \forall i \in [n].\]
From convex optimization, $x^\star = 0$ is preference optimal if (i) $x^\star \in \mathcal{C}$ and (ii) $\mathcal{C}$ is normal to $\nabla f_0$. Indeed, these two conditions can be fulfilled:
\begin{enumerate}
    \item[(i)] Because $v_1,\ldots, v_n$ is assumed to be Pareto generic, zero is a convex combination of the $v_i$'s. As the $z_i$'s are related to the $v_i$'s by a linear transformation, this also implies that zero is a convex combination of the $z_i$'s (with the same set of convex weights). 
    \item[(ii)] We need to show that the subspace $\mathrm{span}(v_1,\ldots, v_n)$ can be mapped into $\mathrm{span}(v_0)^\perp$ by the map $v \mapsto - H^{-1} v$ where $H$ is positive definite. \Cref{lem:positive-def-rotation} shows that such a map $H$ exists as long as $v_0 \notin \mathrm{span}(v_1,\ldots, v_n)$, which is assumed from preference genericity.
\end{enumerate}
Thus, there exists $f_0$ and $F$ that is preference optimal at $x^\star$ with matching first-order information. A necessary stationary condition must therefore be accepted.\hfill $\blacksquare$

\begin{remark} \label{rmk:avoid}
    Suppose that $\mathrm{Stationary}$ is not necessary, but that we can design some optimization method that provably converges to a stationary point in $\{x : \mathrm{Stationary}(x) = \mathsf{true}\}$. Then, this also means that there are settings in which the method provably avoids preference optimal points.
\end{remark}

In the remainder of this section, we prove \Cref{lem:conv-hull} and \Cref{lem:positive-def-rotation} used above.

\begin{lemma} \label{lem:conv-hull}
    Let $f_1,\ldots, f_n : \mathbb{R}^d \to \mathbb{R}$ be positive-definite quadratics with a shared Hessian:
    \[f_i(x) = \frac{1}{2} \|A(x - z_i)\|^2,\]
    where $A \in \mathbb{R}^{d \times d}$ is full rank and $z_i \in \mathbb{R}^d$. Then, the Pareto set is the convex hull:
    \[\mathrm{Pareto}(f_1,\ldots, f_n) = \mathrm{conv}(z_1,\ldots, z_n).\]
\end{lemma}
\begin{proof}
    As the objectives $f_1,\ldots, f_n$ are strongly convex, optimality is equivalent to stationarity. Thus, $x \in \mathrm{Pareto}(f_1,\ldots, f_n)$ if and only if there exists some $\beta \in \Delta^{n-1}$ such that:
    \[0 = \sum_{i \in[n]} \beta_i \nabla f_i(x),\]
    which, when expanded, states that:
    \[(A^\top A)x = (A^\top A) \sum_{i \in [n]} \beta_i z_i.\]
    But as $A$ is invertible, this is equivalent to:
    \[x = \sum_{i \in[n]} \beta_i z_i,\]
    which is to say that $x \in \mathrm{conv}(z_1,\ldots, z_n)$.
\end{proof}

\begin{lemma} \label{lem:positive-def-rotation}
    Let $U$ and $V$ be linear subspaces of $\mathbb{R}^d$ such that $U \cap V^\perp = \{0\}$. Then, there exists some positive definite map $H : \mathbb{R}^d \to \mathbb{R}^d$ such that $H(U) \subset V$.
\end{lemma}

\begin{proof}
    If $S \subset \mathbb{R}^d$ is a subspace, let $\Pi_S : \mathbb{R}^d \to \mathbb{R}^d$ be the projection onto $S$. Define the map:
    \[H := \Pi_V +  \Pi_{V^\perp} \Pi_{U^\perp}.\]
    Then $H$ satisfies the following:
    \begin{itemize}
        \item $H$ is positive definite. To see this, let $0 \ne x \in \mathbb{R}^d$ have decomposition $x = x_1 + x_2$, where $x_1 \in U$ and $x_2 \in U^\perp$. Then:
    \begin{align*}
        x^\top H x &=  \underbrace{x_1^\top \Pi_V x_1 + \textcolor{blue}{2 x_1 \Pi_V x_2} + \textcolor{orange}{x_2 \Pi_V x_2}}_{x^\top \Pi_V x} + \underbrace{\textcolor{blue}{x_1^\top \Pi_{V^\perp}x_2} + \textcolor{orange}{x_2^\top \Pi_{V^\perp}x_2}}_{x^\top \Pi_{V^\perp}\Pi_{U^\perp}x}
        \\&= \|\Pi_V x_1\|^2 + \underbrace{\textcolor{blue}{x_1^\top x_2}}_{0} + \textcolor{blue}{x_1 \Pi_V x_2} + \textcolor{orange}{\|x_2\|^2} \geq \frac{1}{2} \|\Pi_V x_1 + x_2\|^2 > 0,
    \end{align*}
    where the last inequality is strict because $x \ne 0$ and $U \cap V^\perp = \{0\}$. 
    \item $H(U) \subset V$. If $x \in U$, then by definition $\Pi_{U^\perp} x = 0$ so that $Hx = \Pi_V x \in V$.
    \end{itemize}%
\end{proof}

\subsubsection{An example of a first-order stationarity condition avoiding optimality}
In this section, we discuss the first-order stationarity condition of \cite{ye2022pareto}, defined to as stationarity with respect to their optimization dynamics, \emph{Pareto navigating gradient descent} (PNG). We show that it  fails to be a necessary condition for preference optimality.

Despite that, their condition and dynamics have appealing properties since (i) they do not require second-order information, which is computationally more expensive, and (ii) their dynamics largely satisfies what they call the \emph{Pareto improvement property}, which ensures that each objective enjoys monotonic improvement during optimization:
\[\phantom{ \qquad \textrm{for all } i \in [n]}\frac{d}{dt} f_i(x_t) \leq 0, \qquad \textrm{for all } i \in [n].\]
As the goal of Pareto improvement can be at odds with preference optimality, this leads to an open question: when and how should we balance Pareto improvement with preference optimality? 

\begin{definition}[PNG stationarity, \cite{ye2022pareto}]
    Let $c > 0$. Define the \emph{PNG vector} $v_{c}(x)$:
    \begin{align*} 
    v_c(x) := &\argmin_{v \in \mathbb{R}^d}\, \frac{1}{2} \|\nabla f_0(x) - v\|^2 \\
    &\mathrm{s.t.} \ \ \nabla f_i(x)^\top v \geq c, \qquad \textrm{for all }i \in [n].
    \end{align*}
    Let $\epsilon > 0$. A vector $x \in \mathbb{R}^d$ is \emph{$(c,\epsilon)$-PNG stationary} if $v_c(x) = \lambda \nabla f_0(x)$ for some $\lambda \leq 0$ and:
    \[\min_{\beta \in \Delta^{n-1}}\, \|\nabla f_\beta(x)\| = \epsilon.\]
    \label{def:PNG_stationary}
\end{definition}

In the following example, we consider a two-dimensional example with two objectives. Let the standard basis be denoted $\mathbf{e}_1, \mathbf{e}_2 \in \mathbb{R}^2$, and let the objective functions $f_1, f_2 : \mathbb{R}^2 \to \mathbb{R}$ be defined:
\begin{align}
f_1(x) =\frac{1}{2} \|A(x + \mathbf{e}_1)\|^2\qquad\textrm{and}\qquad f_2(x) =\frac{1}{2} \|A(x - \mathbf{e}_1)\|^2,
\label{eq:counter_example_objective}
\end{align}
where $A \in \mathbb{R}^{2 \times 2}$ is full-rank.
\Cref{lem:conv-hull} shows that the Pareto set of the objectives $\mathrm{Pareto}(f_1, f_2)$ is the line segment from $- \mathbf{e}_1$ to $\mathbf{e}_2$. That is, the Pareto set is invariant under changes of $A$. However, the PNG stationarity condition is not, since the constraint set changes with $A$:
\[\big\{v : (x + \mathbf{e}_1)^\top H v \geq c\big\} \cap \big\{v : (x - \mathbf{e}_1)^\top H v \geq c\big\},\]
where $H = A^\top A$. Due to this discrepancy, PNG stationary points can fail to be preference optimal.

\begin{example} \label{ex:png}
    Let the preference function be:
    $f_0(x) =\frac{1}{2} \|x - \mathbf{e}_2\|^2$, and let the objectives $f_1, f_2$ be defined as in the above \Cref{eq:counter_example_objective} with:
    \begin{equation} \label{eqn:PNG-H}
    H = A^\top A = \begin{bmatrix}
        \ 1 & 1 \ \, \\ \ 1 & 2 \ \,
    \end{bmatrix}.
    \end{equation}
    Then, the unique preference optimal point is the origin $0$.
    However, the $(c, \epsilon)$-PNG stationary point is bounded away from 0. 
    It even converges to $\mathbf{e}_1$ as the error tolerance $\epsilon$ goes to zero.
\end{example}
\begin{proof}
    Consider the PNG vector $v_c(x)$ when $x$ is in the region:
    \[\mathcal{C} = \big\{x \in \mathbb{R}^2 : \nabla f_0(x)^\top \nabla f_i(x) < 0,\ \textrm{for }i =1,2\big\} \cap \big\{\mathbf{e}_2^\top x > 0\big\}.\]
    Here, both constraints $\nabla f_i(x)^\top v \geq c$ are active in the constrained optimization problem that defines the PNG vector; and so, $v_c(x)$ is the vertex point of the constraint set, satisfying: 
    \[\nabla f_1(x)^\top v_c(x) = \nabla f_2(x)^\top v_c(x) = c.\]
    Expanding out the gradients, we obtain:
    \[(x + \mathbf{e}_1)^\top H  v_c(x) = c \qquad\textrm{and}\qquad (x - \mathbf{e}_1)^\top H  v_c(x) = c.\]
    This implies that $\mathbf{e}_1^\top Hv_c(x) = 0$. Now suppose that $x_\mathrm{PNG} \in \mathcal{C}$ is PNG stationary. Then, by definition, it must satisfy $\nabla f_0(x_\mathrm{PNG}) \in \mathrm{span}\big(v_c(x_\mathrm{PNG})\big)$, so it has the form:
    \[\phantom{\qquad \textrm{where }\mathbf{e}_1^\top H u = 0} x_\mathrm{PNG} = \mathbf{e}_2 + \lambda u, \qquad \textrm{where }\mathbf{e}_1^\top H u = 0.\]
    Whenever the standard basis vectors are not eigenvectors of $H$, the line $\mathbf{e}_2 + \lambda u$ intersects $\mathrm{Pareto}(f_1, f_2)$ away from 0. In this example, we let $A$ satisfy $H = A^\top A$ where $H$ is given by \Cref{eqn:PNG-H}.
    
    Then, the line $\mathbf{e}_2 + \lambda u$ runs through $\mathbf{e}_1$ and $\mathbf{e}_2$. We can verify that $\mathcal{C}$ contains all points on this line between its two endpoints. When $x = \mathbf{e}_2 + \lambda (\mathbf{e}_1 - \mathbf{e}_2)$ and $\lambda \in (0,1)$, we have:
    \begin{align*}
        \nabla f_0(x)^\top \nabla f_1(x) &= (x - \mathbf{e}_2)^\top H (x + \mathbf{e}_1)
        \\&= \lambda (\mathbf{e}_1 - \mathbf{e}_2)^\top H\big((1 + \lambda) \mathbf{e}_1 + (1-\lambda) \mathbf{e}_2  \big) = - \lambda (1 - \lambda),
    \end{align*}
    and similarly, we have:
    \begin{align*}
        \nabla f_0(x)^\top \nabla f_2(x) &= (x - \mathbf{e}_2)^\top H (x - \mathbf{e}_1)
        \\&= \lambda (\mathbf{e}_1 - \mathbf{e}_2)^\top H\big((\lambda - 1) (\mathbf{e}_1 - \mathbf{e}_2)  \big) = - \lambda (1 - \lambda).
    \end{align*}
    This implies that for all $c > 0$ and $\epsilon > 0$, the $(c, \epsilon)$-PNG stationary point is bounded away from 0, converging to $\mathbf{e}_1$ as $\epsilon$ goes to zero.
\end{proof}

\subsection{Implications of smoothness assumptions}
\lemdiam* 

\paragraph{Proof of \Cref{lem:diam}}
Because each $f_i$ is $\mu$-strongly convex and $L$-Lipschitz smooth, so too is the convex combination $f_\beta$. This implies the upper and lower bounds:
\[\frac{1}{2} \mu \sum_{i \in [n]}  \beta_i \| x - x_i\|_2^2 \leq f_\beta(x) \leq \frac{1}{2} L \sum_{i \in [n]} \beta_i \|x - x_i\|_2^2.\]
It follows that the minimizer of $f_\beta$ is bounded:
\[f_\beta(x_\beta) \leq \frac{1}{2} L r^2.\]
On the other hand, if a point $\|x - x_i \| > 2s$ for some $i \in [n]$, then by reverse triangle inequality, $\|x - x_j\| > s$ for all $j \in [n]$. This implies that:
\[\|x - x_i\| > 2s \qquad \implies \qquad f_\beta(x) > \frac{1}{2} \mu s^2.\]
It follows that if $\|x - x_i\| > 2\sqrt{L/\mu}$ for some $i$, then $x$ is not a Pareto optimal point. \hfill $\blacksquare$

\vspace{11pt}

\xlipschitz*

\paragraph{Proof of \Cref{lem:x*-lipschitz}} That $x^*$ is Lipschitz continuous with Lipschitz continuous gradients follows from the following two lemmas:

\begin{lemma} \label{lem:dx-bound}
    Let $F \equiv (f_1,\ldots, f_n)$ be a set of twice-differentiable objective functions and let $f_0$ be a smooth preference function. Suppose the objectives are $L$-Lipschitz smooth and $\mu$-strongly convex:
    \[ \mu \mathbf{I} \preceq \nabla^2 f_i \preceq L \mathbf{I}.\]
    Let $R := \mathrm{diam}\big(\mathrm{Pareto}(F)\big)$. Then, the map $x^* : (\Delta^{n-1}, \ell_1) \to (\mathbb{R}^d, \ell_2)$ is $LR/\mu$-Lipschitz.
\end{lemma}
\begin{proof}
    Recall from \Cref{eqn:x-best-response-gradient} that $\nabla x^*(\beta) = - \nabla^2 f_\beta(x_\beta)^{-1} \nabla F(x_\beta)^\top$. The following holds:
    \begin{align*} 
        \|\nabla x^*(\beta)\|_{1,2} &\overset{(i)}{\leq} \big\|\nabla^2 f_\beta(x_\beta)^{-1}\big\|_2 \cdot \big\|\nabla F(x_\beta)^\top\big\|_{1,2}
        \\&\overset{(ii)}{\leq} \frac{1}{\mu} \cdot LR,
    \end{align*}
    where (i) is a property of the $\|\cdot \|_{1,2}$-norm, (ii) uses $\mu \mathbf{I} \preceq \nabla^2 f_\beta(x_\beta)$ and \Cref{lem:dF-bound}.
\end{proof}

\begin{lemma} \label{lem:dxb-dxb'}
    Let $\beta, \beta' \in \Delta^{n-1}$. Then,
    \[\big\|\nabla x^*(\beta) - \nabla x^*(\beta')\big\|_{1,2} \leq \frac{2L^2 R}{\mu^2} \left(1 + \frac{L_H R}{\mu}\right) \cdot \|\beta - \beta'\|_1.\]
\end{lemma}
\begin{proof}
    By definition, we have:
    \[\big\|\nabla x^*(\beta) - \nabla x^*(\beta')\big\|_{1,2} = \big\|- \nabla^2 f_\beta(x_\beta)^{-1} \nabla F(x_\beta)^\top  + \nabla^2 f_{\beta'}(x_{\beta'})^{-1} \nabla F(x_{\beta'})^\top \big\|_{1,2}.\]
    We can add and subtract $\nabla^2 f_\beta(x_\beta)^{-1} \nabla F(x_{\beta'})^\top$ inside the norm on the right-hand side (RHS):
    \begin{align*}
        (\mathrm{RHS}) &=  \big\|\textcolor{blue}{- \nabla^2 f_\beta(x_\beta)^{-1} \cdot \big[\nabla F(x_\beta) - \nabla F(x_{\beta'})\big]^\top} + \textcolor{orange}{\big[\nabla^2 f_\beta(x_\beta)^{-1} - \nabla^2 f_{\beta'}(x_{\beta'})^{-1}\big] \cdot \nabla F(x_{\beta'})^\top}\big\|_{1,2}.
    \end{align*}
    We can bound the two terms in the norm separately. For the first:
    \begin{align*}
        \big\|\textcolor{blue}{- \nabla^2 f_\beta(x_\beta)^{-1} \cdot \big[\nabla F(x_\beta) - \nabla F(x_{\beta'})\big]^\top} \big\|_{1,2} \overset{(i)}{\leq} \frac{L}{\mu} \cdot \|x_\beta - x_{\beta'}\| \overset{(ii)}{\leq} \frac{L^2R}{\mu^2} \|\beta - \beta'\|_1,
    \end{align*}
    where (i) follows the same argument as \Cref{lem:approx-dx-bound}, and (ii) applies \Cref{lem:dx-bound}. For the second term, we can add and subtract $\nabla^2 f_{\beta'}(x_\beta)^{-1} \nabla F(x_{\beta'})^\top$ to obtain:
    \begin{align*}
         \big\| \textcolor{orange}{\big[\nabla^2 f_\beta}&\textcolor{orange}{(x_\beta)^{-1} - \nabla^2 f_{\beta'}(x_{\beta'})^{-1}\big] \cdot \nabla F(x_{\beta'})^\top}\big\|_{1,2} 
         \\&=\big\| \big[\textcolor{teal}{\nabla^2 f_\beta(x_\beta)^{-1} - \nabla^2 f_{\beta'}(x_\beta)^{-1}} + \textcolor{purple}{\nabla^2 f_{\beta'}(x_\beta)^{-1} - \nabla^2 f_{\beta'}(x_{\beta'})^{-1}}\big] \cdot \nabla F(x_{\beta'})^\top\big\|_{1,2} 
         \\&\leq \left(\textcolor{teal}{\frac{L}{\mu^2}\|\beta - \beta'\|_1}  + \textcolor{purple}{\frac{L_H}{\mu^2} \frac{L R}{\mu} \|\beta - \beta'\|_1}\right) \cdot L R.
    \end{align*}
    where $\textcolor{teal}{\nabla^2f_\beta(x)^{-1} - \nabla^2f_{\beta'}(x)^{-1}}$ is bounded by \Cref{lem:mat-inv-derivative}; \textcolor{purple}{$\nabla^2 f_{\beta}(x)^{-1} - \nabla^2 f_{\beta}(x')^{-1}$} is bounded by \Cref{lem:mat-inv-lipschitz} and \Cref{lem:dx-bound}; and $\|\nabla F(x_{\beta'})^\top\|_{1,2}$ is bounded by \Cref{lem:dF-bound}.
\end{proof} 
\noindent The result follows by substituting in the definitions of $M_0$ and $M_1$. \hfill $\blacksquare$

\vspace{11pt}

\approxdx*

\begin{proof} 
Recall that $x_\beta := x^*(\beta)$. Then, by definition, we have:
    \begin{align*}
        \big\|\widehat{\nabla} x^*(x,\beta) - \nabla x^*(\beta)\big\|_{1,2} &= \big\|- \nabla^2 f_\beta(x)^{-1} \nabla F(x)^\top  + \nabla^2 f_\beta(x_\beta)^{-1} \nabla F(x_\beta)^\top \big\|_{1,2}.
    \end{align*}
    We can add and subtract $\nabla^2 f_\beta(x)^{-1} \nabla F(x_\beta)^\top$ inside the norm on the right-hand side (RHS) to get:
    \begin{align*}
        (\mathrm{RHS}) &= \big\|\textcolor{blue}{- \nabla^2 f_\beta(x)^{-1} \cdot \big[\nabla F(x) - \nabla F(x_\beta)\big]^\top} + \textcolor{orange}{\big[\nabla^2 f_\beta(x)^{-1} - \nabla^2 f_\beta(x_\beta)^{-1}\big]} \cdot \textcolor{purple}{\nabla F(x_\beta)^\top}\big\|_{1,2}
        \\&\overset{(i)}{\leq} \textcolor{blue}{\frac{L}{\mu}\cdot  \|x - x_\beta\|} + \textcolor{orange}{\frac{L_H}{\mu^2} \|x - x_\beta\|} \cdot \textcolor{purple}{LR} 
        \\&\overset{(ii)}{\leq} \frac{L}{\mu^2} \left( 1 + \frac{L_H R}{\mu}\right) \cdot \|\nabla f_\beta(x)\|,
    \end{align*}
    where (i) the first blue term uses $\mu \mathbf{I} \preceq \nabla^2 f_\beta$ and the $L$-Lipschitz smoothness of the objectives, while the bracket orange term follows from \Cref{lem:mat-inv-lipschitz} and the final purple term follows from \Cref{lem:dF-bound}, and (ii) uses the $\mu$-strong convexity of $f_\beta$.
\end{proof}

\approxgradient*

\paragraph{Proof of \Cref{lem:appro-gradient-bound}} Add and subtract $\nabla f_0(x)^\top \nabla x^*(\beta)$ within the norm on the right-hand side:
\begin{align*}
    (\mathrm{RHS}) &= \left\|\textcolor{blue}{\big(\nabla f_0(x_\beta)^\top - \nabla f_0(x)\big)^\top \nabla x^*(\beta)} + \textcolor{orange}{\nabla f_0(x)^\top \big(\nabla x^*(\beta) - \widehat{\nabla} x^*(x,\beta)\big)}\right\|_{1,2}
    \\&\leq \textcolor{blue}{L_0 M_0\|x_\beta - x\|} + \textcolor{orange}{\|\nabla f_0(x)\| \cdot \frac{1}{\mu} \frac{M_1}{2M_0} \|\nabla f_\beta(x)\|_2},
\end{align*}
where we use the fact that $f_0$ is $L_0$-Lipschitz smooth by \Cref{ass:f0}, that $x^*$ is $M_0$-Lipschitz continuous by \Cref{lem:x*-lipschitz}, and that $\|\textcolor{orange}{\nabla x^*(\beta) - \widehat{\nabla} x^*(\beta)}\|_{1,2}$ is bounded by \Cref{lem:approx-dx-bound}. The result follows from upper bounding $\|x_\beta - x\|$ by $\mu$-strong convexity of $f_\beta$:
\[\|x_\beta - x\|\leq \frac{1}{\mu} \|\nabla f_\beta(x)\|.\]
\hfill $\blacksquare$

\vspace{11pt}

\approxpref*
\paragraph{Proof of \Cref{lem:approx-pref-stationarity-computable}}
    For $(\epsilon, \epsilon_0)$-preference stationarity, we require that $\|\nabla f_{\hat{\beta}}(\hat{x})\|_2 \leq \epsilon$ and:
    \begin{align*}
        \nabla f_0(x_{\hat{\beta}})^\top \nabla x^*(\hat{\beta})(\beta' - \hat{\beta}) + \epsilon_0 \cdot \|\beta' - \hat{\beta}\|_1 \geq 0.
    \end{align*}
    Then by \Cref{lem:appro-gradient-bound}, the left-hand side is lower bounded:
    \begin{align*}
        &\nabla f_0(x)^\top \widehat{\nabla} x^*(x,\hat{\beta}) (\beta' - \hat{\beta}) - \mathrm{err}_{\nabla f_0} (\hat{\beta}, x) \cdot \|\beta' - \hat{\beta}\|_1 + \epsilon_0 \|\beta' - \hat{\beta}\|_1,
        \\&= \underbrace{\nabla f_0(x)^\top \widehat{\nabla} x^*(x,\hat{\beta}) (\beta' - \hat{\beta}) + \alpha \cdot \epsilon_0 \|\beta' - \hat{\beta}\|_1}_{\geq 0}
        + \underbrace{(1 - \alpha) \cdot \epsilon_0 \|\beta' - \hat{\beta}\|_1 - \mathrm{err}_{\nabla f_0} (\hat{\beta}, x) \cdot \|\beta' - \hat{\beta}\|_1}_{\geq 0},
    \end{align*}
    for $\alpha \in (0,1)$. The two terms are lower bounded by zero by conditions (1) and (2), respectively.\hfill $\blacksquare$

\subsection{Convergence for Pareto majorization-minimization}

\pmmconvergence*

\paragraph{Proof of \Cref{thm:pmm-convergence}} 
    Fix $k > 1$. For short, we let:
    \begin{align*}
    (x, \beta) \equiv (x_{k-1}, \beta_{k-1}) \qquad \textrm{and}\qquad (\hat{x} , \hat{\beta}) \equiv (x_k,\beta_k).
    \end{align*}

    \vspace{11pt}
    
    \noindent \textit{Claim.} At each iteration, either (i) the preference improves by at least a constant:
    \[f_0(x_{\hat{\beta}}) - f_0(x_\beta) \leq - \frac{1}{2} \frac{c_1}{\mu_g}\cdot \epsilon_0^2,\] 
    or (ii) the point $(\hat{x},\hat{\beta})$ is $(\epsilon_0, \epsilon)$-preference stationary.

    \vspace{11pt}

    Assuming the claim holds, the theorem immediately follows: if the algorithm in $K$ steps has not found an $(\epsilon_0, \epsilon)$-preference stationary point, then the value $f_0(x_{\beta_k})$ must decrease every iteration by a constant. But because $f_0 \circ x^*$ is lower bounded over $\Delta^{n-1}$ by $f^*$, this can happen at most:
    \[\frac{2 \mu_g\cdot \big(f^* - f_*\big)}{c_1^2 \cdot \epsilon_0^2} \quad \textrm{times}.\]

    \vspace{11pt}

    \noindent \emph{Proof of the claim.} 
    Let $\beta^* := \argmin_{\beta' \in\Delta^{n-1}}\, g(\beta'; x, \beta)$. \Cref{lem:approx-stationary-implies-close} shows that an approximate stationary point $\hat{\beta}$ of a strongly convex function is close to the exact stationary point $\beta^*$:
    \begin{equation} \label{eqn:beta-hat-beta*}
    \|\hat{\beta} - \beta^*\|_2 \leq \frac{c_1\epsilon_0}{\mu_g} =: \delta,
    \end{equation}
    where we let $\delta$ denote this constant for short. 
    
    We can analyze $\hat{\beta}$ through $\beta^*$. There are two cases, leading to either (1) $O(\epsilon_0)$-preference stationarity or (2) $O(\epsilon_0^2)$-constant descent. The two cases depend on the suboptimality of $\beta$.

    \vspace{11pt}
    
    \noindent \underline{Case 1}: $\|\beta^* - \beta\|_2 < 2\delta$. Here, $\beta$ is fairly close to the optimum $\beta^*$ of the surrogate. We show that the approximate stationarity of $\hat{\beta}$ with respect to the surrogate implies approximate preference stationarity. We do so via \Cref{lem:approx-pref-stationarity-computable}, which states that $(\hat{x}, \hat{\beta})$ is $(\epsilon_0, \epsilon)$-preference stationary provided:
    \begin{align}
    \|\nabla f_{\hat{\beta}}(\hat{x})\|_2 &\leq \vphantom{\frac{1}{2}}\epsilon\label{eqn:surrogate-to-df0}\\
    - \nabla f_0(x)^\top \widehat{\nabla} x^*(x, \hat{\beta})(\beta' - \hat{\beta}) &\leq \frac{1}{2} \epsilon_0 \|\beta' - \hat{\beta}\|_1, \qquad \forall \beta' \in \Delta^{n-1} \label{eqn:surrogate-to-stationarity}\\
    \mathrm{err}_{\nabla f_0}(x, \hat{\beta}) &\leq \frac{1}{2} \epsilon_0 \label{eqn:surrogate-to-err}
    \end{align}

    While \Cref{eqn:surrogate-to-df0} is immediate from our choice of $c_2$, defined in the last section of the proof, the others do not follow automatically from approximate stationarity with respect to the surrogate: the surrogate is derived from local information at $(x,\beta)$,  while we would like guarantees at $(x,\hat{\beta})$. But because $\beta^*$ is close to both $\beta$ and $\hat{\beta}$, we can control all of these. By triangle inequality:
    \begin{equation} \label{eqn:hat-beta-beta}
    \|\beta - \hat{\beta}\|_2 \leq  \|\beta - \beta^*\|_2 + \|\beta^* - \hat{\beta}\|_2 < 3\delta,
    \end{equation}
    combining \Cref{eqn:beta-hat-beta*} and the assumption that $\|\beta^* - \beta\|_2 < 2\delta$.

    We now show \Cref{eqn:surrogate-to-stationarity}. We have for all $\beta' \in \Delta^{n-1}$,
    \begin{align*}
        - \nabla f_0(x)^\top &\widehat{\nabla} x^*(x, \hat{\beta})(\beta' - \hat{\beta}) 
        \\&\overset{(i)}{\leq} \textcolor{blue}{- \nabla f_0(x)^\top \widehat{\nabla} x^*(x, \beta)(\beta' -\hat{\beta})} + \|\nabla f_0(x)^\top (\textcolor{orange}{\widehat{\nabla} x^*(x, \hat{\beta}) - \widehat{\nabla} x^*(x, \beta)})\|_{\infty} \cdot \|\beta' -\hat{\beta}\|_1
        \\&\overset{(ii)}{\leq} \textcolor{blue}{c_1 \epsilon_0 \cdot \|\beta' - \hat{\beta}\|_1} + \|\nabla f_0(x)\|_2 \cdot \textcolor{orange}{\frac{L}{\mu^2} \|\beta - \hat{\beta}\|_2} \cdot \|\beta' - \hat{\beta}\|_1
        \\&\overset{(iii)}{\leq} \frac{1}{2}\cdot 2\left( c_1 \epsilon_0 + \frac{L\|\nabla f_0(x)\|_2 }{\mu^2}\cdot 3\delta\right) \cdot \|\beta' - \hat{\beta}\|_1  \numberthis \label{eqn:c1-1}
        \\&\overset{(iv)}{\leq} \frac{1}{2} \epsilon_0 \cdot \|\beta' - \hat{\beta}\|_1,
    \end{align*}
    where (i) adds and subtracts $\nabla f_0(x)^\top \widehat{\nabla} x^*(x,\beta)(\beta' - \hat{\beta})$ and applies H\"older's inequality, (ii) substitutes in Condition 1 for the first term and bounds the second via \Cref{lem:mat-inv-derivative}, and (iii) bounds $\|\beta - \hat{\beta}\|_2$ using \Cref{eqn:hat-beta-beta}, and (iv) applies the definition of $c_1$, set in the last section of the proof.

    To show \Cref{eqn:surrogate-to-err}, we have:
    \begin{align*}
        \mathrm{err}_{\nabla f_0}(x,\hat{\beta}) &\overset{(i)}{=} \mathrm{err}_{\nabla f_0}(x,\beta) + \frac{1}{\mu}\left( \frac{M_1}{2M_0} \|\nabla f_0(x)\|_2 + L_0 M_0\right)  \left(\|\nabla f_{\hat{\beta}}(x)\|_2 - \|\nabla f_{\beta}(x)\|_2\right)
        \\&\overset{(ii)}{\leq} \mathrm{err}_{\nabla f_0}(x,\beta) + \frac{1}{\mu}\left( \frac{M_1}{2M_0} \|\nabla f_0(x)\|_2 + L_0 M_0\right)  \|\nabla F(x)^\top\|_{2} \cdot \|\hat{\beta} - \beta\|_2
        \\&\overset{(iii)}{\leq} \frac{1}{2}\cdot \frac{2}{\mu}\left( \frac{M_1}{2M_0} \|\nabla f_0(x)\|_2 + L_0 M_0\right) \left\{ c_2 \epsilon +   \|\nabla F(x)^\top\|_{2} \cdot 3\delta\right\} \numberthis \label{eqn:c1-2}
        \\&\overset{(iv)}{\leq} \frac{1}{2} \epsilon_0.
    \end{align*}
    where (i) expands out $\mathrm{err}_{\nabla f_0}$, (ii) uses the fact that $\beta \mapsto \|\nabla F(x)^\top \beta\|_2$ is $\|\nabla F(x)^\top\|_2$-Lipschitz in $\beta$ with respect to the $\ell_2$-norm, (iii) applies the definition of $\mathrm{err}_{\nabla f_0}$ and the inequality \Cref{eqn:hat-beta-beta}, and (iv) follows by definition of $c_1$ and $c_2$, set in the last section of the proof.

    As \Cref{eqn:surrogate-to-df0,eqn:surrogate-to-stationarity,eqn:surrogate-to-err} hold, \Cref{lem:approx-pref-stationarity-computable} shows that $(\hat{x},\hat{\beta})$ is $(\epsilon_0, \epsilon)$-preference stationary.

    \vspace{11pt}

    \noindent \underline{Case 2}:  $\|\beta^* - \beta\|_2 \geq 2 \delta$. Here $\beta$ is suboptimal and $\beta^*$ achieves a large descent:
    \begin{align*} 
    f_0(x_{\beta^*}) - f_0(x_\beta) &\overset{(i)}{\leq} g(\beta^*; x,\beta) - f_0(x_\beta)
    \\&\overset{(ii)}{\leq} \mathrm{err}_{\nabla f_0}(x,\beta) - \frac{1}{2} \mu_g \|\beta^* - \beta\|_2^2 
    \\&\overset{(iii)}{\leq} \frac{1}{\mu}\left( \frac{M_1}{2M_0} \|\nabla f_0(x)\|_2 + L_0 M_0\right) c_2\epsilon_0^2 -  2\mu_g \delta^2 \numberthis \label{eqn:c2-2}
    \\&\overset{(iv)}{\leq} - \frac{3}{2} \mu_g \delta^2, \numberthis \label{eqn:g-beta*}
    \end{align*}
    where (i) uses the majorizing property of $g$, (ii) follows from \Cref{lem:Q-descent}, (iii) applies the definition of $\mathrm{err}_{\nabla f_0}(x,\beta)$ along with the assumption that $\epsilon \leq \epsilon_0^2$, and (iv) applies the definition of $c_2$.
    
    The large descent also carries over to $\hat{\beta}$ because it is approximately stationary:
    \begin{align*}
        f_0(x_{\hat{\beta}}) - f_0(x_\beta) &\overset{(i)}{\leq} g(\hat{\beta}; x, \beta) - f_0(x_\beta) 
        \\&\overset{(ii)}{=} g(\beta^*; x, \beta) - f_0(x_\beta) + \big(g(\hat{\beta}; x, \beta) - g(\beta^*; x, \beta)\big)
        \\&\overset{(iii)}{\leq} - \frac{3}{2} \mu_g \delta^2 + c_1 \epsilon_0 \cdot \delta = - \frac{1}{2} \frac{c_1}{\mu_g}\cdot \epsilon_0^2,
    \end{align*}
    where (i) uses the majorizing property of $g$, (ii) adds and subtracts $g(\beta^*; x, \beta)$ and (iii) applies \Cref{eqn:g-beta*} and \Cref{lem:approx-stationary-implies-close}.

    Thus, the preference improves by at least a constant. To finish proving the claim, we need to verify that it is indeed possible to set $c_1$ and $c_2$ appropriately.

    \vspace{11pt}

    \noindent \underline{Setting $c_1$ and $c_2$}: we tabled a few inequalities above. Recall:

    For \Cref{eqn:surrogate-to-df0}, we need:
    \[c_2 \leq 1.\]
    
    For \Cref{eqn:c1-1}, we need:
    \[2\left( c_1 \epsilon_0 + \frac{3 L\|\nabla f_0(x)\|_2 }{\mu^2}\cdot \frac{c_1\epsilon_0}{\mu_g}\right) \leq \epsilon_0.\]
    
    For \Cref{eqn:c1-2}, we need:
    \[\frac{2}{\mu}\left( \frac{M_1}{2M_0} \|\nabla f_0(x)\|_2 + L_0 M_0\right) \left\{ c_2\epsilon +   3 \|\nabla F(x)^\top\|_{2} \cdot \frac{c_1\epsilon_0}{\mu_g}\right\} \leq \epsilon_0.\]

    For \Cref{eqn:c2-2}, we need:
    \[\frac{1}{\mu}\left( \frac{M_1}{2M_0} \|\nabla f_0(x)\|_2 + L_0 M_0\right) c_2 \epsilon_0^2 \leq \frac{1}{2} \mu_g \left(\frac{c_1\epsilon_0}{\mu_g}\right)^2.\]
    It is unenlightening but straightforward to verify that it suffices to set:
    \begin{align*} 
    &c_1 \cdot \max\left\{2 + \frac{6L \|\nabla f_0(x)\|_2}{\mu^2 \cdot \mu_g}, \frac{12}{\mu \cdot \mu_g}\left( \frac{M_1}{2M_0} \|\nabla f_0(x)\|_2 + L_0 M_0\right) \cdot \|\nabla F(x)^\top\|_2 \right\} \leq 1\\
    &c_2 \cdot \max\left\{1, \frac{2}{\mu}\left( \frac{M_1}{2M_0} \|\nabla f_0(x)\|_2 + L_0 M_0\right) \cdot \left(2 \vee \frac{\mu_g}{c_1^2} \right)\right\} \leq 1,
    \end{align*}
    where $a \vee b := \max\{a,b\}$.

    A concerned reader may wonder whether $c_1$ and $c_2$ may be bounded away from zero, as claimed in the theorem statement: we need to ensure that $\|\nabla f_0(x)\|_2$ and $\|\nabla F(x)^\top\|_2$ do not blow up. Indeed, this holds because the iterates $x_k$ remain within a constant distance of the Pareto set. In particular, since $c_2 \leq 1$, by Condition 2, we have that the $k$th iterate satisfies:
    \[\|x_k - x_{\beta_k}\| \leq \frac{\epsilon}{\mu},\]
    which follows from $\mu$-strong convexity of $f_{\beta_k}$. Thus, all iterates of the algorithm are within $\epsilon/\mu$ of the Pareto set and also satisfy for all $k, k' \in \mathbb{N}$:
    \[\|x_k - x_{k'}\|\leq R + 2 \epsilon / \mu.\] 
    Then, by $L_0$-Lipschitz smoothness, we can bound:
    \begin{align*} 
    \|\nabla f_0(x_k)\| &\leq \|\nabla f_0(x_1)\| + \|\nabla f_0(x_k) - \nabla f_0(x_1)\| 
    \\&\leq \|\nabla f_0(x_1)\| + L_0 \cdot (R + 2 \epsilon / \mu).
    \end{align*}
    Similarly, by $L$-Lipschitz smoothness, we also have:
    \begin{align*} 
    \|\nabla F(x_k)^\top\|_2 &\leq \|\nabla F(x_1)^\top\|_2 + \|\nabla F(x_k)^\top - \nabla F(x_1)^\top\|_2 
    \\&\leq \|\nabla F(x_1)^\top\|_2 + n L \cdot (R + 2 \epsilon / \mu).
    \end{align*}
\hfill$\blacksquare$

\subsubsection{Analytic lemma: gradient bound}

\begin{lemma} \label{lem:dF-bound}
    Let $R := \mathrm{diam}\big(\mathrm{Pareto}(F)\big)$. Then for any $x_\beta = x^*(\beta)$,
    \[\big\|\nabla F(x_\beta)^\top\big\|_{1,2} \leq LR.\]
\end{lemma}
\begin{proof}
    By definition, we have:
    \begin{align*}
        \big\|\nabla F(x_\beta)^\top\big\|_{1,2} &= \sup_{\|z\|_1 = 1}\, \bigg\|\sum_{i\in [n]} z_i \nabla f_i(x_\beta)\bigg\|_2
        \\&\overset{(i)}{\leq} \sup_{\|z\|_1 = 1} \, \sum_{i\in[n]} |z_i| \cdot \|\mathrm{sign}(z_i) \cdot \nabla f_i(x_\beta)\|_2
        \\&\overset{(ii)}{\leq} \max_{i \in [n]}\, \|\nabla f_i(x_\beta)\|_2 
        \\&\overset{(iii)}{\leq} \max_{i \in [n]}\, L \|x - x_i\|_2,
    \end{align*}
    where (i) follows from Jensen's inequality, (ii) holds because the max is no smaller than the average, (iii) applies $L$-Lipschitz smoothness. In particular, let $x_i = \argmin\, f_i(x)$, so that $\nabla f_i(x_i) = 0$. Then:
    \[\|\nabla f_i(x_\beta) - \nabla f_i(x_i)\|_2 \leq L \|x_\beta - x_i\|_2.\]
    The result holds because $x_\beta$ and all $x_i$'s are contained in $\mathrm{Pareto}(F)$.
\end{proof}

\subsection{Analytic lemmas: matrix inverses}

\begin{lemma} \label{lem:mat-inv-lipschitz}
    Let $M : \mathbb{R}^d \to \mathbb{R}^{d \times d}$ be $L$-Lipschitz satisfying $M(x) \succeq \mu \mathbf{I}$ where $\mathbb{R}^d$ has the $\ell_2$-norm and $\mathbb{R}^{d \times d}$ the operator norm. Then, the map $x \mapsto M(x)^{-1}$ is $L / \mu^2$-Lipschitz.
\end{lemma}
\begin{proof}
    For short, let us denote $M(x)$ by $M_x$. Note that $\mathbf{I} = (M_x' + M_x - M_x')M_x^{-1}$, so that:
    \begin{align*}
        M_x^{-1} - M_{x'}^{-1} &= M_x^{-1} - M_{x'}^{-1}\big(\textcolor{blue}{M_{x'}} + \textcolor{orange}{M_x - M_{x'}}\big) M_x^{-1}
        \\&= M_x^{-1} - \textcolor{blue}{M_x^{-1}}  - M_{x'}^{-1} \big(\textcolor{orange}{M_x - M_{x'}}\big) M_{x}^{-1} = - M_{x'}^{-1}\big(M_x - M_{x'}\big) M_x^{-1},
    \end{align*}
    which is series of unenlightening algebraic manipulations. But now, we may apply $L$-Lipschitz continuity to obtain $\|M_x - M_{x'}\| \leq L \|x - x'\|$ and the $\mu$-lower bound to obtain $\|M_{x}^{-1}\|, \|M_{x'}^{-1}\| \leq \mu^{-1}$. Together, we obtain $L/\mu^2$-Lipschitz continuity:
    \[\big\|M(x)^{-1} - M(x')^{-1}\big\| \leq \frac{L}{\mu^2} \|x - x'\|.\]
\end{proof}

\begin{lemma} \label{lem:mat-inv-derivative}
    Let $ M_1,\ldots, M_n$ be positive-definite matrices in $\mathbb{R}^{d\times d}$ equipped with the operator norm, and let $\Delta^{n-1}$ be equipped with the $\ell_1$ norm. Suppose the following holds:
    \[\mu\mathbf{I} \preceq M_1,\ldots, M_n \preceq L \mathbf{I}.\]
    Then, the map $\beta \mapsto M_\beta^{-1}$ where $M_\beta := \sum_{i \in [n]} \beta_i M_i$ has bounded derivative $\|\nabla_\beta M_\beta^{-1}\|_{1,2} \leq L/\mu^2$.
\end{lemma}
\begin{proof}
    We can compute the derivative of the above map:
    \[\nabla_\beta M_\beta^{-1} = - M_\beta^{-1} \big(\nabla_\beta M_\beta\big) M_\beta^{-1},\]
    where $\nabla_\beta M_\beta d\beta = M_{d\beta}$. The upper bound on the $M_i$'s implies that $\|\nabla_\beta M_\beta\|_\mathrm{1,2} \leq L$. And on the other hand, the lower bound implies that $\|M_\beta^{-1}\|_2 \leq \mu^{-1}$.
\end{proof}

\subsubsection{Analytic lemmas: constrained optimization of strongly convex functions}

\begin{lemma} \label{lem:approx-stationary-implies-close}
    Let $f : \mathbb{R}^n \to \mathbb{R}$ be smooth and convex and let $\mathcal{C} \subset \mathbb{R}^n$ be a convex constraint set. Suppose that $\beta^*, \hat{\beta} \in \mathcal{C}$ are stationary and $\epsilon$-approximately stationary, respectively:
    \[\phantom{ \quad \forall \beta \in \mathcal{C}}- \nabla f(\beta^*)^\top(\beta - \beta^*) \leq 0\qquad \textrm{and}\qquad - \nabla f(\hat{\beta})^\top(\beta - \hat{\beta}) \leq \epsilon \|\beta - \hat{\beta}\|, \quad \forall \beta \in \mathcal{C}.\]
    Then, $f(\hat{\beta}) - f(\beta^*) \leq \epsilon \|\hat{\beta} - \beta^*\|$. Furthermore, if $f$ is $\mu$-strongly convex, then $\|\hat{\beta} - \beta^*\| \leq \epsilon/\mu$. 
\end{lemma}
\begin{proof}
    For the first part, we apply the mean value theorem, which states that there exists some $\beta$ that is a convex combination of $\hat{\beta}$ and $\beta^*$ such that:
    \begin{align*} 
    f(\hat{\beta}) - f(\beta^*) &\overset{(i)}{=} \nabla f(\beta)^\top (\hat{\beta} - \beta^*)
    \\&\overset{(ii)}{\leq} \nabla f(\hat{\beta})^\top (\hat{\beta} - \beta^*)
    \\&\overset{(iii)}{\leq} \epsilon \|\hat{\beta} - \beta^*\|,
    \end{align*}
    where (i) applies the mean value theorem, (ii) uses the monotonicity of gradients of convex functions:
    \begin{align*}
        \big(\nabla f(\hat{\beta}) - \nabla f(\beta)\big)^\top(\hat{\beta} - \beta) \geq 0
    \end{align*}
    and that $\hat{\beta} - \beta = \lambda (\hat{\beta} - \beta^*)$ for some $\lambda \in [0,1]$, and (iii) applies the $\epsilon$-stationarity condition.
    
    For the second part, by strong convexity, we have on the one hand:
    \[\big(\nabla f(\hat{\beta}) - \nabla f(\beta^*)\big)^\top (\hat{\beta} - \beta^*) \geq \mu \|\hat{\beta} - \beta^*\|^2. \]
    And on the other, by stationarity and $\epsilon$-stationarity, we have that:
    \[\big(\nabla f(\hat{\beta}) - \nabla f(\beta^*)\big)^\top (\hat{\beta} - \beta^*) \geq \epsilon \|\hat{\beta} - \beta^*\|.\]
    Dividing through by $\|\hat{\beta} - \beta^*\|$ yields the result.
\end{proof}

\begin{restatable}{lemma}{qdescent}\label{lem:Q-descent}
    Let $\mathcal{C} \subset \mathbb{R}^n$ be a convex constraint set with $\beta \in \mathcal{C}$, and let $Q : \mathcal{C} \to \mathbb{R}$ be a quadratic: 
    \begin{equation}  \label{eqn:Q}
    Q(\beta') = c + v^\top (\beta' - \beta) + \frac{1}{2} C \|\beta' - \beta\|^2,
    \end{equation}
    where $c \in \mathbb{R}$, $v \in \mathbb{R}^n$, and $C > 0$. Let $\beta^* \in \mathcal{C}$ minimize $Q$. If $\|\beta^* - \beta\| \geq \epsilon > 0$, then:
    \[Q(\beta^*) - Q(\beta) \leq - \frac{1}{2} C \epsilon^2.\]
\end{restatable}

\begin{proof}
    Define the quadratic function $q : \mathbb{R} \to \mathbb{R}$ by:
    \begin{align} 
    q(\lambda) &= c + \lambda v^\top (\beta^* - \beta) + \frac{1}{2} C \lambda^2 \|\beta^* - \beta\|^2 \notag
    \\&= c + \frac{1}{2} C\|\beta^* - \beta\|^2 \lambda \big(\lambda - 2\lambda^*\big) \label{eqn:q-lambda}
    \end{align}
    where $\lambda^* = - \frac{v^\top (\beta^* - \beta)}{C\|\beta^* - \beta\|^2}$ minimizes $q$. Restricting $Q$ to the line between $\beta$ and $\beta^*$, we get:
    \[Q\big(\beta + \lambda( \beta^* - \beta)\big) = q(\lambda),\]
    for $\lambda \in [0,1]$. This follows by expanding the definition of $Q$. 
    
    Notice that $q$ monotonically decreases on the interval $0 \leq \lambda \leq \lambda^*$, and also that $q$ monotonically increases for $\lambda > \lambda^*$. Because $Q(\beta^*) = q(1)$ minimizes $Q$ on the convex set $\mathcal{C}$, $q$ must be descending on $\lambda \in [0,1]$. Thus, $1 \leq \lambda^*$. It follows that $1 - 2 \lambda^* \leq -1$. Plugging in into \Cref{eqn:q-lambda}, we have:
    \[Q(\beta^*) = q(1) \leq c - \frac{1}{2} C\|\beta^* - \beta\|^2.\]
    Applying $Q(\beta_0) = c$ and $\|\beta^* - \beta\| \geq \epsilon$ yields the result.
\end{proof}